\documentclass[psamsfonts, reqno]{amsart}

\usepackage{amsmath}
\usepackage{amssymb}
\usepackage{amsthm}
\usepackage{amscd}
\usepackage[margin=1in]{geometry}
\usepackage{fancyhdr}
\usepackage{todonotes}
\usepackage{setspace}
\usepackage[all,arc]{xy}
\usepackage{enumerate}
\usepackage{mathrsfs}
\usepackage{color}
\usepackage{tikz-cd}
\usepackage{hyperref}
\usepackage{yhmath}

\usepackage{xcolor}

\newtheorem{thm}{Theorem}[section]
\newtheorem{cor}[thm]{Corollary}
\newtheorem{prop}[thm]{Proposition}
\newtheorem{lem}[thm]{Lemma}

\newtheorem*{thm*}{Theorem}
\newtheorem*{prop*}{Proposition}
\newtheorem*{lem*}{Lemma}
\newtheorem*{cor*}{Corollary}

\newtheorem{lem'}{Lemma}
\newtheorem{thm'}{Theorem}
\newtheorem{prop'}{Proposition}

\theoremstyle{definition}

\newtheorem{exmp}[thm]{Example}
\newtheorem*{exmp*}{Example}

\newtheorem*{exer*}{Exercise}

\theoremstyle{remark}
\newtheorem{rem}[thm]{Remark}

\newtheorem*{claim*}{Claim}

\newtheorem*{rem*}{Remark}
\newtheorem{rem'}{Remark}
\newtheorem{qtn}[thm]{Question}
\newtheorem*{qtn'}{Question}
\newtheorem*{soln'}{Solution}

\DeclareMathOperator{\supp}{supp}

\DeclareMathOperator{\im}{Im}

\DeclareMathOperator{\Div}{div}

\DeclareMathOperator{\Ker}{Ker}
\DeclareMathOperator{\Coker}{Coker}
\DeclareMathOperator{\Dirac}{Dirac}
\DeclareMathOperator{\esssupp}{ess\,supp}

\renewcommand{\d}{\,\textnormal{d}}

\newcommand{\RomanNumeralCaps}[1]{\MakeUppercase{\romannumeral #1}}

\numberwithin{equation}{section}

\title{Density Criteria for Fourier Uniqueness Phenomena in $\mathbf{R}^d$}

\author{Anshul Adve}
\address{Department of Mathematics, Princeton University, Princeton, NJ 08540, USA}
\email{aadve@princeton.edu}

\begin{document}
	
	\begin{abstract}
		We show that if a closed discrete subset $A \subseteq \mathbf{R}^d$ is denser than a certain critical threshold, then $A$ is a Fourier uniqueness set, while if $A$ is sparser, then uniqueness fails and one can prescribe arbitrary values for a Schwartz function and its Fourier transform on $A$ (assuming those values are rapidly decreasing). More general results of the same nature hold for Fourier uniqueness pairs. This is an analog in all dimensions of the work \cite{KNS} of Kulikov, Nazarov, and Sodin in dimension $1$. Our methods are unrelated. As an application of our results, we produce Fourier uniqueness sets in higher dimensions which are optimally well-separated (up to constants).
		Our techniques also give the first purely analytic construction of discrete Fourier uniqueness pairs in higher dimensions. For a concrete example, consider
		\begin{align*}
			A = \{\delta |n|^{t-1} n : n \in \mathbf{Z}^d\}
			\qquad \text{and} \qquad
			B = \{\delta |n|^{u-1} n : n \in \mathbf{Z}^d\},
		\end{align*}
		where $t,u,\delta > 0$ and $t+u = 1$. We show that when $\delta$ is sufficiently small, $(A,B)$ is a Fourier uniqueness pair, but when $\delta$ is sufficiently large, there is an infinite-dimensional space of Schwartz functions $f$ with $f|_A = \hat{f}|_B = 0$.
	\end{abstract}
	
	\maketitle
	
	\tableofcontents
	
	\parskip 0.5em
	
	
	\section{Introduction}
	
	\subsection{Motivating question} In 2017, Radchenko and Viazovska \cite{Rad_Viaz} proved a remarkable \emph{Fourier interpolation formula}, which expresses any even Schwartz function $f$ on $\mathbf{R}$ in terms of the restrictions $f|_{\{\sqrt{n}\}}$ and $\hat{f}|_{\{\sqrt{n}\}}$ of $f$ and its Fourier transform to the discrete set $\{\sqrt{n} : n \in \mathbf{Z}^{\geq 0}\}$. They showed in addition that $f$ and $\hat{f}$ can take any (rapidly decreasing) values on $\{\sqrt{n}\}$, so long as the Poisson summation formula is satisfied.
	This was a prototype for analogous Fourier interpolation formulas for radial Schwartz functions in dimensions $8$ and $24$. These formulas were the key ingredient in their proof of universal optimality of the $E_8$ and Leech lattices, together with Cohn, Kumar, and Miller \cite{CKMRV}. All of these interpolation formulas involve modular forms, and are highly dependent on the arithmetic structure of $\{\sqrt{n}\}$.
	
	A corollary in \cite{Rad_Viaz} is that if $f \in \mathcal{S}_{\text{even}}(\mathbf{R})$ satisfies $f|_{\{\sqrt{n}\}} = \hat{f}|_{\{\sqrt{n}\}} = 0$, then $f = 0$. One says that $\{\sqrt{n}\}$ is a \emph{Fourier uniqueness set for even Schwartz functions} on $\mathbf{R}$. More generally, a pair $(A,B)$ of subsets $A,B \subseteq \mathbf{R}^d$ is a \emph{Fourier uniqueness pair} if $f|_A = 0$ and $\hat{f}|_B = 0$ implies $f = 0$ for all Schwartz functions $f \in \mathcal{S}(\mathbf{R}^d)$. A subset $A \subseteq \mathbf{R}^d$ is a \emph{Fourier uniqueness set} (for all Schwartz functions) if $(A,A)$ is a Fourier uniqueness pair. Fourier uniqueness seems to be an analytic property rather than an arithmetic property, so it is natural to expect that the density of $A,B$, rather than the structure of $A,B$, is the primary factor determining whether or not $(A,B)$ is a Fourier uniqueness pair. Our results make this precise.
	
	Given closed discrete subsets $A,B \subseteq \mathbf{R}^d$, let $\mathcal{F}_{A,B} \colon \mathcal{S}(\mathbf{R}^d) \to \mathcal{S}(A) \oplus \mathcal{S}(B)$ denote the map $f \mapsto (f|_A, \hat{f}|_B)$ from Schwartz functions on $\mathbf{R}^d$ to rapidly decreasing sequences on $A,B$. Then $\mathcal{F}_{A,B}$ being injective means that $(A,B)$ is a Fourier uniqueness pair, and $\mathcal{F}_{A,B}$ being surjective means that one can prescribe arbitrary values for $f$ on $A$ and $\hat{f}$ on $B$. This paper addresses the following question.
	
	\begin{qtn} \label{qtn:main}
		How do the densities of $A,B$ influence the injectivity / surjectivity properties of $\mathcal{F}_{A,B}$?
	\end{qtn}
	
	In the case $A = B$, our main results may be summarized in two bullet points:
	\begin{itemize}
		\item If $A$ is denser than $\{|n|^{-1/2}n : n \in \mathbf{Z}^d\}$, then $\mathcal{F}_{A,A}$ is injective but not surjective.
		
		\item If $A$ is sparser than $\{|n|^{-1/2}n : n \in \mathbf{Z}^d\}$, then $\mathcal{F}_{A,A}$ is surjective but not injective.
	\end{itemize}
	We have analogous results in general when $A$ and $B$ are not necessarily the same.
	When $d=1$, the set $\{|n|^{-1/2}n\}$ is just $\{\pm \sqrt{n}\}$, which is consistent with \cite{Rad_Viaz}.
	
	Most of this is known in dimension $1$ by unpublished work of Kulikov, Nazarov, and Sodin \cite{KNS}. Their method for proving Fourier uniqueness relies on the one-dimensional phenomenon that a closed discrete subset of $\mathbf{R}$ partitions $\mathbf{R}$ into intervals. We develop new techniques that work in all dimensions.
	
	
	\subsection{Main results} \label{subsec:main_results}
	
	For $p \geq 0$ and $\delta > 0$, a subset $S \subseteq \mathbf{R}^d$ is \emph{$(p,\delta)$-dense} (resp. \emph{$(p,\delta)$-separated}) if the balls $B_{\delta \langle x \rangle^{-p}}(x)$ centered at $x \in S$ cover $\mathbf{R}^d$ (resp. are disjoint). Here $\langle x \rangle$ is the Japanese bracket
	\begin{align} \label{eqn:Japanese_bracket}
		\langle x \rangle = \sqrt{1+|x|^2} \sim \max\{1,|x|\}
	\end{align}
	(see Section \ref{subsubsec:estimates_notn} for the definition of $\sim$ and similar notation).
	
	\begin{exmp} \label{exmp:(p,delta)}
		The set
		\begin{align*}
			\{|n|^{-\frac{p}{p+1}} n : n \in \mathbf{Z}^d\}
		\end{align*}
		in $\mathbf{R}^d$ is $(p,\delta)$-dense for $\delta$ sufficiently large and $(p,\delta)$-separated for $\delta$ sufficiently small. When $p = 1$ and $d = 1$, this set is $\{\pm \sqrt{n}\}$.
	\end{exmp}
	
	
	\begin{thm}[Subcritical case]\label{thm:sub_non-end}
		Let $p,q,\delta > 0$ with $pq < 1$. Let $A,B \subseteq \mathbf{R}^d$ be $(p,\delta)$-separated and $(q,\delta)$-separated, respectively. Then $\Coker \mathcal{F}_{A,B}$ is finite-dimensional with dimension $O_{p,q,d,\delta}(1)$, and $\Ker \mathcal{F}_{A,B}$ is infinite-dimensional. If in addition $\delta$ is sufficiently large depending on $p,q,d$, then $\mathcal{F}_{A,B}$ is surjective.
	\end{thm}
	
	\begin{thm}[Supercritical case]\label{thm:super_non-end}
		Let $p,q,\delta > 0$ with $pq > 1$. Let $A,B \subseteq \mathbf{R}^d$ be $(p,\delta)$-dense and $(q,\delta)$-dense, respectively. Then $\Ker \mathcal{F}_{A,B}$ is finite-dimensional with dimension $O_{p,q,d,\delta}(1)$, and $\Coker \mathcal{F}_{A,B}$ is infinite-dimensional. If in addition $\delta$ is sufficiently small depending on $p,q,d$, then $\mathcal{F}_{A,B}$ is injective.
	\end{thm}
	
	It is no loss in generality to use the same $\delta$ for both $A$ and $B$, because $\mathcal{F}_{A,B}$ is conjugate to $\mathcal{F}_{\lambda A, \lambda^{-1}B}$ for any $\lambda > 0$ by rescaling, so we can replace $A,B$ with $\lambda A, \lambda^{-1}B$.
	
	In the critical case $pq = 1$, weaker versions of Theorems \ref{thm:sub_non-end} and \ref{thm:super_non-end} still hold.
	
	\begin{thm}[Sparse critical case]\label{thm:sub_end}
		Let $p,q,\delta > 0$ with $pq = 1$ and $\delta$ sufficiently large depending on $p,q,d$. Let $A,B \subseteq \mathbf{R}^d$ be $(p,\delta)$-separated and $(q,\delta)$-separated, respectively. Then $\Ker \mathcal{F}_{A,B}$ is infinite-dimensional.
	\end{thm}
	
	For the dense critical case, we need one more definition. We say that a diffeomorphism $\sigma \colon \mathbf{R}^d \to \mathbf{R}^d$ is \emph{$r$-pseudohomogeneous} for $r > 0$ if it satisfies the derivative bounds
	\begin{gather*}
		\langle \sigma(x) \rangle \sim \langle x \rangle^r, \\
		|\nabla^k \sigma(x)| \lesssim_k \langle x \rangle^{r-k}, \\
		|D\sigma(x)^{-1}| \sim \langle x \rangle^{1-r}
	\end{gather*}
	for all $k \in \mathbf{N}$ (here $|D\sigma(x)^{-1}|$ is the operator norm of the inverse of the Jacobian matrix of $\sigma$ at $x$). For example, any diffeomorphism which is homogeneous of order $r$ outside a bounded set is $r$-pseudohomogeneous (with the constants in the derivative bounds depending on $\sigma$, of course). Conversely, given a homogeneous diffeomorphism of $\mathbf{R}^d \setminus \{0\}$ of order $r > 0$ whose action on spheres is isotopic to an orthogonal linear transformation, it can be modified near the origin to extend to a pseudohomogeneous diffeomorphism of $\mathbf{R}^d$.
	
	If $\sigma$ is $\frac{1}{p+1}$-pseudohomogeneous for $p \geq 0$, then Taylor expansion shows that $\sigma(\mathbf{Z}^d)$ is $(p,\delta)$-dense for $\delta$ large and $(p,\delta)$-separated for $\delta$ small. This generalizes Example \ref{exmp:(p,delta)}.
	%
	
	\begin{thm}[Dense critical case]\label{thm:super_end}
		Let $p,q > 0$ with $pq = 1$. Let $\sigma_A,\sigma_B \colon \mathbf{R}^d \to \mathbf{R}^d$ be diffeomorphisms which are $\frac{1}{p+1}$-pseudohomogeneous and $\frac{1}{q+1}$-pseudohomogeneous, respectively. Put
		\begin{align}\label{eqn:A,B_structure}
			A = \{\delta \sigma_A(n) : n \in \mathbf{Z}^d\}
			\qquad \text{and} \qquad
			B = \{\delta \sigma_B(n) : n \in \mathbf{Z}^d\},
		\end{align}
		where $\delta > 0$ is sufficiently small depending on $p,q,d$ (and on the derivative bounds on $\sigma_A,\sigma_B$). Then $\mathcal{F}_{A,B}$ is injective, and $\Coker \mathcal{F}_{A,B}$ is infinite-dimensional.
	\end{thm}
	
	
	A sample application of Theorems \ref{thm:sub_end} and \ref{thm:super_end} is
	
	\begin{cor} \label{cor:powers}
		Let $p,q,\delta > 0$ with $pq = 1$. Put
		\begin{align*}
			A = \{\delta |n|^{-\frac{p}{p+1}} n : n \in \mathbf{Z}^d\}
			\qquad \text{and} \qquad
			B = \{\delta |n|^{-\frac{q}{q+1}} n : n \in \mathbf{Z}^d\}.
		\end{align*}
		If $\delta$ is sufficiently small depending on $p,q,d$, then $\mathcal{F}_{A,B}$ is injective. On the other hand, if $\delta$ is sufficiently large depending on $p,q,d$, then $\Ker \mathcal{F}_{A,B}$ is infinite-dimensional.
	\end{cor}
	
	With the change of variables $(t,u) = (\frac{1}{p+1},\frac{1}{q+1})$, this is the result advertised in the abstract. Note that $pq = 1$ if and only if $t+u = 1$.
	
	\subsection{Previous work} \label{subsec:previous_work}
	
	As pointed out in \cite{Rad_Viaz}, the classical Nyquist--Shannon sampling theorem \cite{Nyquist, Shannon} is a Fourier uniqueness result: it says that $(\mathbf{Z}, \mathbf{R} \setminus (-\frac{1}{2},\frac{1}{2}))$ is a Fourier uniqueness pair in dimension $1$. There is a corresponding explicit interpolation formula which predates Nyquist and Shannon --- it was proven by Whittaker in 1915 \cite{Whittaker}. Over a century later, the first Fourier uniqueness pair $(A,B)$ with both $A,B$ discrete was discovered by Radchenko and Viazovska in \cite{Rad_Viaz}. This initiated the modern study of Fourier uniqueness and interpolation, leading to work of both algebraic \cite{BRS, CKMRV, Rad_Stol, Sardari, Stoller, Viaz} and analytic \cite{KNS, Ramos_Sousa, Ramos_Stol} flavor.
	
	Progress toward Question \ref{qtn:main} began with \cite{Ramos_Sousa}, in which Ramos and Sousa showed that $\{\pm n^t\}$ is a Fourier uniqueness set in dimension $1$ for $0 < t < 1 - \frac{\sqrt{2}}{2}$ (this is the diagonal case of their more general result for pairs $(\{\pm n^t\}, \{\pm n^u\})$). The threshold $1-\frac{\sqrt{2}}{2} \approx 0.29$ isn't sharp, and the full range $0 < t < \frac{1}{2}$ was later obtained by Kulikov, Nazarov, and Sodin as a consequence of their work \cite{KNS}, which comprehensively answers the injectivity half of Question \ref{qtn:main} in dimension $1$.
	
	In most respects, our results in dimension $1$ are less precise than those of Kulikov--Nazarov--Sodin. They have a similar notion of subcritical and supercritical density (note that their $p,q$ corresponds to our $p+1,q+1$), and they show that a supercritical pair is a Fourier uniqueness pair, whereas $\Ker \mathcal{F}_{A,B}$ is infinite-dimensional for a subcritical pair. For them, the density gap between subcritical and supercritical is a factor of $1-\varepsilon$ (for any $\varepsilon > 0$), while for us, the density gap is a factor of $O_{\varepsilon}(R^{\varepsilon})$ at distance $\sim R$ from the origin. Consequently, our Theorems \ref{thm:sub_end} and \ref{thm:super_end} fall into their subcritical and supercritical cases, respectively. In addition, their results imply a more precise version of our Corollary \ref{cor:powers} in dimension~$1$: there exists an explicit $\delta_{\text{crit}} > 0$ depending on $p,q$ such that if $A = \{\pm \delta n^{\frac{1}{p+1}}\}$ and $B = \{\pm \delta n^{\frac{1}{q+1}}\}$, then $(A,B)$ is a Fourier uniqueness pair when $\delta < \delta_{\text{crit}}$, and $\Ker \mathcal{F}_{A,B}$ is infinite-dimensional when $\delta > \delta_{\text{crit}}$. We believe the analogous statement holds in all dimensions, but we are unsure what the exact value of $\delta_{\text{crit}}$ should be. Another difference between their work and ours is that they only need to consider the asymptotic density of $A,B$ near infinity. Therefore in our Theorem \ref{thm:super_non-end} when $d = 1$, it follows from \cite{KNS} that the size of $\delta$ is irrelevant, so $\Ker \mathcal{F}_{A,B}$ is always zero.
	
	The only thing we prove in dimension $1$ that doesn't appear in \cite{KNS} is that $\Coker \mathcal{F}_{A,B}$ is finite-dimensional in our subcritical case, with surjectivity when $\delta$ is large. This is true, and new, in all dimensions.
	
	Far less was previously known in dimension $>1$. In the dense case, there are various results \cite{CKMRV, Ramos_Sousa, Stoller} under a spherical symmetry assumption (either $A,B$ must be unions of spheres, hence not discrete, or $f$ must be radial), but these are not so different from the one-dimensional setting, as the radial variable plays the main role. Higher-dimensional Fourier uniqueness pairs with $A,B \subseteq \mathbf{R}^d$ discrete were first obtained by Ramos and Stoller \cite{Ramos_Stol} by perturbing a Fourier interpolation formula for functions restricted to spheres of radii $\sqrt{n}$. They in particular produced discrete Fourier uniqueness sets $A = B$. These sets are extremely dense, containing super-exponentially many points in the ball of radius $R$, as $R \to \infty$. Soon after, Viazovska \cite{Viaz} constructed discrete Fourier uniqueness sets by using similar ideas to \cite{Stoller}, but taking $A,B$ to be unions of spherical designs \cite{BRV, DGS} rather than unions of spheres. This gave examples with $R^{O(1)}$ many points in the ball of radius $R$. The exponent was not made explicit in \cite{Viaz}, but it is much larger than necessary. Taking $p = q = 1$ in Corollary \ref{cor:powers}, we obtain extremely simple Fourier uniqueness sets with $\sim R^{2d}$ many points in the ball of radius $R$. Theorems \ref{thm:sub_non-end} and \ref{thm:sub_end} strongly suggest that the exponent $2d$ is optimal. Certainly by Theorem \ref{thm:sub_end}, these Fourier uniqueness sets are optimally well-separated, up to constants.
	
	Our proof of Fourier uniqueness in Theorems \ref{thm:super_non-end} and \ref{thm:super_end} is by real-variable analytic techniques, in contrast with all previous work in dimension~$>1$.
	
	
	
	In the sparse case in higher dimensions, Radchenko and Stoller \cite{Rad_Stol} showed that if $A$ is the ``component-wise square root" of a lattice $\Lambda \subseteq \mathbf{R}^d$ coming from a certain number-theoretic construction ($\Lambda$ should be the inverse different of a totally real number field of degree $d$), then $\Ker \mathcal{F}_{A,A}$ is infinite-dimensional. This $A$ also has $\sim R^{2d}$ many points in the ball of radius $R$. In view of Theorems \ref{thm:super_non-end} and \ref{thm:super_end}, a set $A'$ with $\Ker \mathcal{F}_{A',A'}$ infinite-dimensional cannot be much denser than $A$. On the other hand, Theorem \ref{thm:sub_end} yields a large family of sets $A'$ of density within a constant factor of the density of $A$, such that $\Ker \mathcal{F}_{A',A'}$ is infinite-dimensional for the sole reason that $A'$ is sufficiently well-separated.
	
	\subsection{Notation, conventions, and terminology} \label{subsec:notation}
	
	We use notation and language that is common in harmonic analysis and PDE. Since this paper may be of interest to people from other areas, in Sections \ref{subsubsec:estimates_notn} through \ref{subsubsec:essential} we carefully describe the language we need. First, some general notation.
	
	\subsubsection{General notation}
	
	Our normalization of the Fourier transform is
	\begin{align*}
		\hat{f}(\xi) = \int_{\mathbf{R}^d} f(x) e^{-2\pi ix \cdot \xi} \d x.
	\end{align*}
	The inverse Fourier transform is then
	\begin{align*}
		g^{\vee}(x) = \int_{\mathbf{R}^d} g(\xi) e^{2\pi ix \cdot \xi} \d\xi.
	\end{align*}
	
	The natural numbers $\mathbf{N} = \{0,1,2,\dots\}$ include zero. A \emph{dyadic number} is an integer power of $2$. The set of dyadic numbers is denoted $2^{\mathbf{Z}}$. The subset of dyadic numbers $\geq 1$ is $2^{\mathbf{N}}$.
	
	Recall from \eqref{eqn:Japanese_bracket} that the Japanese bracket $\langle x \rangle$ is defined to be $\sqrt{1+|x|^2}$, which is a smooth function comparable to $\max\{1,|x|\}$.
	
	The open ball of radius $R$ around $x_0 \in \mathbf{R}^d$ is denoted $B_R(x_0)$.
	
	Given $E \subseteq F \subseteq \mathbf{R}^d$ and $\delta > 0$, the subset $E$ is \emph{$\delta$-dense} in $F$ if the balls $B_{\delta}(x)$ for $x \in E$ cover $F$. When $F = \mathbf{R}^d$, this is the same as $E$ being $(0,\delta)$-dense as defined in Section \ref{subsec:main_results}.
	
	\subsubsection{Estimates} \label{subsubsec:estimates_notn}
	
	Throughout, we allow all constants to depend on $p,q,d$, where $p,q,d$ are as in the theorems. So for nonnegative quantities $X,Y \geq 0$, the notation $X \lesssim_{a,b,\dots} Y$ or equivalently $Y \gtrsim_{a,b,\dots} X$ means that $X \leq CY$, where $C$ is a constant depending only on $p,q,d$, and the parameters $a,b,\dots$ (we occasionally include $p,q,d$ in the list of parameters for emphasis). We write $X \sim_{a,b,\dots} Y$ when both $X \lesssim_{a,b,\dots} Y$ and $X \gtrsim_{a,b,\dots} Y$. The expression $X \ll_{a,b,\dots} Y$ means $X \leq cY$, where $c$ is a constant which is sufficiently small for the relevant context, but which is not too small in the sense that $c \gtrsim_{a,b,\dots} 1$. It is typical to split an argument into cases based on whether $X \ll_{a,b,\dots} Y$ or $X \gtrsim_{a,b,\dots} Y$. The strongest type of inequality is $X \lll Y$, which denotes the informal assertion that $X$ is smaller than $Y$ by much more than a constant factor. We only use $X \lll Y$ for heuristics, never in a formal proof. When $X \in \mathbf{C}$ and $Y \geq 0$, the big-$O$ notation $X = O_{a,b,\dots}(Y)$ means $|X| \lesssim_{a,b,\dots} Y$.
	
	For $h$ nonnegative, $h^{\infty}$ denotes any quantity which is $\lesssim_j h^j$ for all $j \in \mathbf{N}$. For example, a smooth function $f$ on $\mathbf{R}^d$ is Schwartz if and only if it satisfies the derivative bounds
	\begin{align} \label{eqn:Schwartz_def}
		|\nabla^k f(x)| \lesssim_{f,k} \langle x \rangle^{-\infty}
	\end{align}
	for all $k \in \mathbf{N}$. Here $\nabla^k f$ is the tuple of $k$th partial derivatives. When $k = 0$, this is just $f$, and when $k = 1$, it is the gradient. When $k > 1$, there is no natural order on the components of $\nabla^k f$, so we only write $\nabla^k f$ when the order doesn't matter (e.g., if $\nabla^k f$ is in absolute values as in \eqref{eqn:Schwartz_def}). 
	
	
	\subsubsection{Adapted Schwartz and bump functions} \label{subsubsec:adapted}
	
	A \emph{standard Schwartz function} is one which obeys the derivative bounds
	\begin{align} \label{eqn:std_Schwartz_def}
		|\nabla^k f(x)| \lesssim_k \langle x \rangle^{-\infty}.
	\end{align}
	The point of the definition is that unlike in \eqref{eqn:Schwartz_def}, the constants here are not allowed to depend on $f$, so an infinite family of standard Schwartz functions satisfies uniform derivative bounds. A \emph{standard bump function} is a standard Schwartz function which is compactly supported in a ball around the origin of radius $\lesssim 1$. So $f \in C_c^{\infty}(\{|x| \lesssim 1\})$ is a standard bump function if and only if
	\begin{align*}
		|\nabla^k f(x)| \lesssim_k 1.
	\end{align*}
	A \emph{Schwartz (resp. bump) function adapted to a ball} $B_R(x_0)$ in $\mathbf{R}^d$ is a function $f$ of the form
	\begin{align} \label{eqn:adapted_basic}
		f(x) = \chi\Big(\frac{x-x_0}{R}\Big),
	\end{align}
	where $\chi$ is a standard Schwartz (resp. bump) function. More generally, an \emph{$L^p$-normalized Schwartz (resp. bump) function adapted to $B_R(x_0)$} is of the form
	\begin{align*}
		f(x) = R^{-d/p} \chi\Big(\frac{x-x_0}{R}\Big)
	\end{align*}
	with $\chi$ a standard Schwartz (resp. bump) function. So $\|f\|_{L^p} \lesssim 1$. If $L^p$-normalization is not specified, then by default $f$ is $L^{\infty}$-normalized; this agrees with the definition \eqref{eqn:adapted_basic}. Given $x_0,\xi_0 \in \mathbf{R}^d$ and $R > 0$, an \emph{$L^p$-normalized wave packet adapted to $B_R(x_0) \times B_{R^{-1}}(\xi_0)$ in phase space} is a function $f$ on $\mathbf{R}^d$ of the form
	\begin{align*} 
		f(x) = R^{-d/p} \chi\Big(\frac{x-x_0}{R}\Big) e^{2\pi i\xi_0 \cdot x},
	\end{align*}
	where $\chi$ is a standard Schwartz function (see, e.g., \cite{Folland} for a reference on phase space). Again, $p = \infty$ is the default. The Fourier transform $\hat{f}$ is an $L^{p'}$-normalized wave packet adapted to $B_{R^{-1}}(\xi_0) \times B_{R}(-x_0)$, where $\frac{1}{p} + \frac{1}{p'} = 1$. This illustrates that, in a sense, the Fourier transform acts on phase space by rotation by $-90$ degrees: $(x,\xi) \mapsto (\xi,-x)$.
	If $R^{-1} \gtrsim |\xi_0|$, then $f$ is indistinguishable from an $L^p$-normalized Schwartz function adapted to $B_R(x_0)$. Similarly if $R \gtrsim |x_0|$, then $\hat{f}$ is indistinguishable from an $L^{p'}$-normalized Schwartz function adapted to $B_{R^{-1}}(\xi_0)$. This is a manifestation of the uncertainty principle (see \cite{Tao_phase} for a very nice discussion).
	In general, a \emph{wave packet} is a function
	\begin{align} \label{eqn:wave_packet_def}
		f(x) = A \chi\Big(\frac{x-x_0}{R}\Big) e^{2\pi i\xi_0 \cdot x}
	\end{align}
	for some $x_0,\xi_0 \in \mathbf{R}^d$, $A,R > 0$, and standard Schwartz $\chi$.
	
	\subsubsection{Essential support} \label{subsubsec:essential}
	
	The \emph{essential support}, also called the \emph{essential physical support}, of a function $f$, denoted $\esssupp f$, is the region where most of its mass lives. This is a somewhat fuzzy notion, but in practice it is usually clear what it means. For example, the essential support of \eqref{eqn:adapted_basic} is $\{|x-x_0| \lesssim R\}$. The \emph{essential Fourier support} of $f$ is the essential support of $\hat{f}$. One can also talk about the \emph{essential phase space support} of $f$, which is the set of $(x,\xi) \in \mathbf{R}^d \times \mathbf{R}^d$ such that near $x$, $f$ oscillates at frequencies near $\xi$. The most important example is a wave packet \eqref{eqn:wave_packet_def}, which has essential physical support $\{|x-x_0| \lesssim R\}$, essential Fourier support $\{|\xi-\xi_0| \lesssim R^{-1}\}$, and essential phase space support the Cartesian product of these two balls. The uncertainty principle says that the essential phase space support of a function $f$ always has volume $\gtrsim 1$ in $\mathbf{R}^d \times \mathbf{R}^d$. This is sharp for wave packets. An upper bound (i.e., a superset) for the essential phase space support of $f$ may be found by writing $f$ as a sum of wave packets and using the subadditivity property
	\begin{align*}
		\esssupp(g+h)
		\subseteq \esssupp g \cup \esssupp h
	\end{align*}
	(which we have written for physical essential support but which makes sense for Fourier and phase space essential support too). In general, we use the word ``essential" to indicate that one should ignore rapid decay in whatever we're about to say. This is analogous to the use of ``essential" in measure theory to indicate that measure zero sets should be ignored.
	
	\subsection{Outlines of proofs} This section is more informal than the rest of the paper.
	
	We present the sparse setting (Theorems \ref{thm:sub_non-end} and \ref{thm:sub_end}) first and the dense setting (Theorems \ref{thm:super_non-end} and \ref{thm:super_end}) second. This is roughly in order of difficulty. These two settings are independent of each other.
	
	\subsubsection{Subcritical case (Theorem \ref{thm:sub_non-end})} \label{subsubsec:sub_non-end}
	
	The main challenge is to show that $\mathcal{F}_{A,B}$ is surjective whenever $\delta$ is sufficiently large. Once this is established, the rest of the theorem will follow by some simple reductions (carried out in Section \ref{subsec:completing_sub_non-end}). So assume $\delta$ is large.
	
	If $A',B'$ are any translates of $A,B$, then $\mathcal{F}_{A,B}$ is conjugate to $\mathcal{F}_{A',B'}$, so $\mathcal{F}_{A,B}$ is surjective if and only if $\mathcal{F}_{A',B'}$ is surjective. Therefore by translating $A,B$, we may assume all points in $A,B$ have magnitude $\gtrsim \delta^{\varpi}$, where $\varpi \gtrsim 1$ is a fixed small constant.
	
	Let $(\alpha,\beta) \in \mathcal{S}(A) \oplus \mathcal{S}(B)$. We want to find a Schwartz function $f \in \mathcal{S}(\mathbf{R}^d)$ solving the equations $f|_A = \alpha$ and $\hat{f}|_B = \beta$. We do this by an iterative procedure, constructing a sequence of approximate solutions $f_n$ converging to $f$. Suppose for illustration that $\beta = 0$ (this is no loss in generality, because if we can handle the case $\beta = 0$, then by Fourier symmetry we can handle the case $\alpha = 0$, and then by linearity we can handle the general case $(\alpha,\beta) = (\alpha,0) + (0,\beta)$). Let $\{|a| \lesssim R\}$ be a large ball on which $\alpha$ is essentially supported.
	
	For each $a_0 \in A$, let $\varphi_{a_0}$ be a Schwartz function adapted to the ball $\{|x-a_0| \lesssim |a_0|^{-p}\}$, with $\varphi_{a_0}(a_0) = 1$. Then since $A$ is $(p,\delta)$-separated and $\delta$ is large, $\varphi_{a_0}(a) \approx 1_{a = a_0}$ for $a \in A$ (the error in $\approx$ comes from Schwartz tails; it can be avoided by taking $\varphi_{a_0}$ to be compactly supported, but it doesn't matter). The essential Fourier support of $\varphi_{a_0}$ is $\{|\xi| \lesssim |a_0|^p\}$. For each $b_0 \in B$, similarly let $\psi_{b_0}$ be a Schwartz function such that $\widehat{\psi_{b_0}}$ is adapted to the ball $\{|\xi-b_0| \lesssim |b_0|^{-q}\}$, and $\widehat{\psi_{b_0}}(b_0) = 1$. Then $\widehat{\psi_{b_0}}(b) \approx 1_{b=b_0}$ for $b \in B$, and $\psi_{b_0}$ has essential physical support $\{|x| \lesssim |b_0|^q\}$.
	
	Define the first approximate solution by
	\begin{align*}
		f_1 = \sum_{a_0 \in A} \alpha(a_0) \varphi_{a_0}.
	\end{align*}
	Then $f_1|_A \approx \alpha$, but on the Fourier side the approximation may be bad, i.e., $\widehat{f_1}|_B \not\approx 0$. Since $\alpha$ is essentially supported on $\{|a| \lesssim R\}$, the essential Fourier support of $f_1$ is
	\begin{align} \label{eqn:ess_supp_f_1_hat}
		\esssupp \widehat{f_1}
		\subseteq \bigcup_{|a_0| \lesssim R} \esssupp \widehat{\varphi_{a_0}}
		= \{|\xi| \lesssim R^p\}.
	\end{align}
	We construct $f_2$ to fix the error on the Fourier side:
	\begin{align*}
		f_2 = f_1 - \sum_{b_0 \in B} \widehat{f_1}(b_0) \psi_{b_0}.
	\end{align*}
	Then $\widehat{f_2}|_B \approx 0$, but the approximation on the physical side is ruined. It seems at first glance that we have made no progress. However, we have won something on the physical side: the essential support of the error $\alpha - f_2|_A$ has shrunk. Indeed, using $\alpha-f_2|_A \approx \sum \widehat{f_1}(b_0) \psi_{b_0}|_A$ and \eqref{eqn:ess_supp_f_1_hat},
	\begin{align*}
		\esssupp (\alpha-f_2|_A)
		\subseteq \bigcup_{|b_0| \lesssim R^p} \esssupp \psi_{b_0}
		= \{|x| \lesssim R^{pq}\}.
	\end{align*}
	Since we are in the subcritical case $pq < 1$, the right hand side is much smaller than our upper bound $\{|a| \lesssim R\}$ for the essential support of $\alpha$. This allows us to iterate.
	
	Having defined $f_{2k}$, recursively define
	\begin{align*}
		f_{2k+1} = f_{2k} + \sum_{a_0 \in A} (\alpha-f_{2k})(a_0) \varphi_{a_0}
		\qquad \text{and} \qquad
		f_{2k+2} = f_{2k+1} - \sum_{b_0 \in B} \widehat{f_{2k+1}}(b_0) \psi_{b_0}.
	\end{align*}
	By induction,
	\begin{align*}
		\widehat{f_{2k}}|_B \approx 0
		\qquad \text{and} \qquad
		\esssupp(\alpha-f_{2k}|_A) \subseteq \{|a| \lesssim R^{(pq)^k}\},
	\end{align*}
	and also
	\begin{align*}
		f_{2k+1}|_A \approx \alpha
		\qquad \text{and} \qquad
		\esssupp \widehat{f_{2k+1}}|_B \subseteq \{|b| \lesssim R^{p(pq)^k}\}.
	\end{align*}
	These essential supports eventually end up within a ball around the origin of bounded radius, and hence at distance $\gtrsim \delta^{\varpi}$ from $A,B$. Once this happens, the approximation $\mathcal{F}_{A,B} f_n$ gets closer to $(\alpha,0)$ pointwise with each iteration, and consequently the change $f_{n+1}-f_n$ gets smaller. We therefore expect that, in some suitable sense, the $f_n$ converge to a function $f$ with $\mathcal{F}_{A,B}f = (\alpha,0)$.
	
	For the rigorous proof, it is convenient to rephrase this iteration procedure in terms of a parametrix $P$ for the operator $\mathcal{F}_{A,B}$, i.e., an approximate right inverse $P \colon \mathcal{S}(A) \oplus \mathcal{S}(B) \to \mathcal{S}(\mathbf{R}^d)$ of $\mathcal{F}_{A,B}$. Let
	\begin{align}\label{eqn:sub_P_def}
		P(\alpha,\beta)
		= \sum_{a_0 \in A} \alpha(a_0) \varphi_{a_0} + \sum_{b_0 \in B} \beta(b_0) \psi_{b_0}.
	\end{align}
	The series
	\begin{align}\label{eqn:FP_series_inverse}
		\sum_{j=0}^{\infty} (1-\mathcal{F}_{A,B}P)^{j}
	\end{align}
	formally inverts $\mathcal{F}_{A,B}P$, so if we can show that it converges in the strong operator topology on the space of continuous linear operators on $\mathcal{S}(A) \oplus \mathcal{S}(B)$ (where $\mathcal{S}(A)$ and $\mathcal{S}(B)$ have the natural Fr\'echet topology), then we can conclude that $\mathcal{F}_{A,B}$ is surjective. In the above notation,
	\begin{align*}
		f_n \approx P \sum_{j=0}^{n-1} (1-\mathcal{F}_{A,B}P)^j (\alpha,0)
	\end{align*}
	(again, this $\approx$ becomes an exact equality if $\varphi_{a_0}$ and $\widehat{\psi_{b_0}}$ are chosen to be compactly supported). Thus we see that proving the convergence of \eqref{eqn:FP_series_inverse} amounts to making the above heuristics precise.
	
	\subsubsection{Sparse critical case (Theorem \ref{thm:sub_end})} \label{subsubsec:sub_end}
	
	We aim to prove that $\Ker \mathcal{F}_{A,B}$ is infinite-dimensional. As in the subcritical case (see Section \ref{subsec:completing_sub_non-end}), this would follow if we could show that  $\mathcal{F}_{A',B'}$ is surjective, where $A',B'$ are constructed from $A,B$ by adding a thin sequence of points, $A'$ is $(p,\delta')$-separated, $B'$ is $(q,\delta')$-separated, and $\delta' \gtrsim \delta$. Rename $A = A'$, $B = B'$, and $\delta = \delta'$. Then $A,B$ satisfy the assumptions of Theorem \ref{thm:sub_end}, and it suffices to show that $\mathcal{F}_{A,B}$ is surjective. Actually, it is clear from the argument in Section \ref{subsec:completing_sub_non-end} that surjectivity onto $\mathcal{S}(A) \oplus \mathcal{S}(B)$ is stronger than necessary; it is enough if all finitely supported sequences are in the image of $\mathcal{F}_{A,B}$. We will show the intermediate statement that the image of $\mathcal{F}_{A,B}$ contains $\mathcal{T}(A,B)$, where $\mathcal{T}(A,B)$ is a Banach space of sequences which decay at a specific super-polynomial rate (so $\mathcal{T}(A,B) \hookrightarrow \mathcal{S}(A) \oplus \mathcal{S}(B)$).
	
	The strategy to do this is basically the same as in the subcritical case. We define $P$ by the same formula \eqref{eqn:sub_P_def}, and we check that the series \eqref{eqn:FP_series_inverse} converges in the norm topology on the space of bounded linear operators on $\mathcal{T}(A,B)$. The only technical detail is that we need to be a little more careful about how we choose the functions $\varphi_{a_0},\psi_{b_0}$. First, we let $\varphi_{a_0}$ and $\widehat{\psi_{b_0}}$ be Schwartz functions adapted to the balls
	\begin{align*}
		\{|x-a_0| \lesssim \delta^{1/2} |a_0|^{-p}\}
		\qquad \text{and} \qquad
		\{|\xi-b_0| \lesssim \delta^{1/2} |b_0|^{-q}\},
	\end{align*}
	respectively. The extra factor of $\delta^{1/2}$ here makes
	\begin{align}\label{eqn:sub_end_phi_psi_supp}
		\esssupp \widehat{\varphi_{a_0}}
		= \{|\xi| \lesssim \delta^{-1/2} |a_0|^p\}
		\qquad \text{and} \qquad
		\esssupp \psi_{b_0}
		= \{|x| \lesssim \delta^{-1/2} |b_0|^q\}
	\end{align}
	smaller by a factor of $\delta^{-1/2}$. This $\delta^{-1/2}$ ensures that we still have the phenomenon of ``shrinking essential support" that we exploited in the subcritical case.
	
	Second, we choose $\widehat{\varphi_{a_0}}$ and $\psi_{b_0}$ to be compactly supported on the regions \eqref{eqn:sub_end_phi_psi_supp}. Otherwise, unlike in the subcritical case, the Schwartz tails of $\widehat{\varphi_{a_0}}$ and $\psi_{b_0}$ would cause problems. Note that once we do this, we no longer have the freedom to take $\varphi_{a_0}$ and $\widehat{\psi_{b_0}}$ to be compactly supported. Therefore, the $\approx$'s in the analysis in Section \ref{subsubsec:sub_end} cannot be replaced by $=$'s. Fortunately, this is not an issue.
	
	Third, we take $\widehat{\varphi_{a_0}}$ and $\psi_{b_0}$ to be Gevrey \cite{Rodino}, meaning that $\varphi_{a_0}$ and $\widehat{\psi_{b_0}}$ have quasi-exponential decay (see \eqref{eqn:quasi-exp_decay} and \eqref{eqn:phi_psi_Gevrey}). This is necessary to match the fast decay of sequences in $\mathcal{T}(A,B)$.
	
	\subsubsection{Supercritical case (Theorem \ref{thm:super_non-end})} \label{subsubsec:super_non-end}
	
	Similarly to the sparse setting, the main difficulty is to show that $\mathcal{F}_{A,B}$ is injective whenever $\delta$ is sufficiently small. So assume $\delta$ is small.
	
	Let $\mathcal{M}_{A,B} \subseteq \mathcal{S}'(\mathbf{R}^d)$ denote the closed linear span of tempered measures supported on $A$ and inverse Fourier transforms of tempered measures supported on $B$. To prove injectivity of $\mathcal{F}_{A,B}$, it suffices by duality to show that $\mathcal{M}_{A,B} = \mathcal{S}'(\mathbf{R}^d)$, or in other words, that the inclusion $\mathcal{M}_{A,B} \hookrightarrow \mathcal{S}'(\mathbf{R}^d)$ is surjective. We will construct a parametrix $P$ for this inclusion, so $P$ is a map taking $f$ on $\mathbf{R}^d$ to $Pf \in \mathcal{M}_{A,B}$, where $Pf \approx f$. Iterating the parametrix will then give $\mathcal{M}_{A,B} = \mathcal{S}'(\mathbf{R}^d)$, and hence that $\mathcal{F}_{A,B}$ is injective.
	
	To make the approximation $\mathcal{M}_{A,B} \ni Pf \approx f$, we use the Gabor inversion formula (see Section \ref{subsec:Gabor}) to decompose $f$ in a phase space basis $(\chi_{y,\eta})_{y,\eta \in \mathbf{R}^d}$, where $\chi_{y,\eta}$ is a wave packet adapted to $B_1(y) \times B_1(\eta)$ in phase space. By linearity, it then suffices to approximate each $\chi_{y,\eta}$ by some $\nu_{y,\eta} \in \mathcal{M}_{A,B}$. To simplify matters, let us suppose for now that $p,q > 1$ (all the ideas appear in this case). Suppose further that $\langle y \rangle \geq \langle \eta \rangle$ (the case $\langle y \rangle < \langle \eta \rangle$ is symmetric via the Fourier transform). Then we will take $\nu_{y,\eta}$ to simply be a measure on $A$ (when $\langle y \rangle < \langle \eta \rangle$ we take $\nu_{y,\eta}$ to be the inverse Fourier transform of a measure on $B$, respecting the Fourier symmetry). Near $y$, the set $A$ is $O(\delta \langle y \rangle^{-p})$-dense, so it is reasonable to hope that we can choose the measure $\nu_{y,\eta}$ to be a good approximation to $\chi_{y,\eta}$ at frequency scales $\ll \delta^{-1} \langle y \rangle^p$ (at higher frequency scales, the uncertainty principle prevents us from making a good approximation). After solving one technical problem, described below, we will indeed be able to find such a $\nu_{y,\eta}$. Since $p > 1$ and $\langle y \rangle \geq \langle \eta \rangle$, we have $\langle \eta \rangle \lll \delta^{-1} \langle y \rangle^p$, so $\chi_{y,\eta}$ has essential Fourier support at frequencies much smaller than $\delta^{-1} \langle y \rangle^p$. Therefore $\nu_{y,\eta}$ will capture the full behavior of $\chi_{y,\eta}$, not just the low-frequency behavior. However, again by the uncertainty principle, $\nu_{y,\eta}$ must itself have high-frequency behavior. From these considerations (plus the fact that $\nu_{y,\eta}$ will have the same essential physical support as $\chi_{y,\eta}$), we expect the error $\nu_{y,\eta} - \chi_{y,\eta}$ to have essential phase space support in the region
	\begin{align} \label{eqn:nu-chi_portrait}
		\{(x,\xi) \in \mathbf{R}^d \times \mathbf{R}^d : |x-y| \lesssim 1 \text{ and } |\xi| \gtrsim \delta^{-1} \langle y \rangle^p\},
	\end{align}
	with essentially all of the mass coming from the high frequency behavior of $\nu_{y,\eta}$.
	
	Now, the crucial observation is that because $\langle \eta \rangle \lll \delta^{-1} \langle y \rangle^p$, this region \eqref{eqn:nu-chi_portrait} is much further from the origin than the essential phase space support of $\chi_{y,\eta}$. Ultimately, this means that when we iterate the parametrix to produce a sequence of approximations
	\begin{align*}
		f_n = P \sum_{j=0}^{n-1} (1-P)^j f \in \mathcal{M}_{A,B}
	\end{align*}
	to $f$, the error $f-f_n$ will have essential phase space support going out to infinity as $n \to \infty$. This phenomenon --- the opposite of the ``shrinking essential support" phenomenon in the sparse setting --- is why $f_n \to f$ as tempered distributions.
	
	The technical problem referred to above is that in the approximation $\nu_{y,\eta} \approx \chi_{y,\eta}$, we in some sense need an arbitrary power saving error term at low frequencies. The naive way to approximate a function by a discrete measure is to use Riemann sums; this gives a power saving error, but not an arbitrary power.
	In dimension $1$, one can do better by using the trapezoid rule instead of Riemann sums. The trapezoid rule is further improved by Simpson's rule. In Section \ref{subsec:sampling}, we formulate and prove a generalization of these sorts of numerical integration techniques. This will allow us to produce a sufficiently good approximation $\nu_{y,\eta} \approx \chi_{y,\eta}$.
	
	The general case where one of $p,q$ may be $\leq 1$ is almost entirely the same, except the condition $\langle y \rangle \geq \langle \eta \rangle$ looks a little more complicated (see \eqref{eqn:nu_def_non-end}).
	
	\subsubsection{Dense critical case (Theorem \ref{thm:super_end})} \label{subsubsec:super_end}
	
	Again, the primary challenge is to show injectivity.
	It follows from \eqref{eqn:A,B_structure} that the sets $A,B$ are $(p,O(\delta^{p+1}))$-dense and $(q,O(\delta^{q+1}))$-dense, respectively.
	First consider the case $p = q = 1$ where $A,B$ have balanced density (the general case is discussed below).
	Our strategy is the same as in Section \ref{subsubsec:super_non-end}.
	By duality and linearity, we reduce to approximating the wave packet $\chi_{y,\eta}$ by some $\nu_{y,\eta} \in \mathcal{M}_{A,B}$, for each $y,\eta \in \mathbf{R}^d$. By Fourier symmetry, we reduce further to $\langle y \rangle \geq \langle \eta \rangle$. Then we take $\nu_{y,\eta}$ to be a measure on $A$. Unlike in the supercritical case where we had made no assumption on the structure of $A$, here we can describe $\nu_{y,\eta}$ explicitly in terms of $\sigma_A$, and avoid using Section \ref{subsec:sampling}. The analysis in Section \ref{subsubsec:super_non-end} shows that the essential support of $\nu_{y,\eta} - \chi_{y,\eta}$ in phase space is contained in
	\begin{align}\label{eqn:nu-chi_portrait_end_trivial}
		\{(x,\xi) \in \mathbf{R}^d \times \mathbf{R}^d : |x-y| \lesssim 1 \text{ and } |\xi| \gtrsim \delta^{-2} \langle y \rangle\}
	\end{align}
	(this agrees with \eqref{eqn:nu-chi_portrait} because $A$ is $(1,O(\delta^2))$-dense). If we try to mimic the arguments from the supercritical case using only this information about $\nu_{y,\eta} - \chi_{y,\eta}$, then the errors accumulate too quickly when iterating the parametrix. This is because the region \eqref{eqn:nu-chi_portrait_end_trivial} is only a constant factor (depending on $\delta$) further from the origin than the essential support of $\chi_{y,\eta}$ in phase space, which is not enough to compensate for the fact that \eqref{eqn:nu-chi_portrait_end_trivial} is much larger than the essential phase space support of $\chi_{y,\eta}$. The portion of \eqref{eqn:nu-chi_portrait_end_trivial} with $|\xi| \sim \delta^{-2} \langle y \rangle$ has volume $\sim \delta^{-2d} \langle y \rangle^d$, and the factor of $\langle y \rangle^d$ kills us. We need the intersection of $\{|\xi| \lesssim \delta^{-2} \langle y \rangle\}$ with the essential phase space support of $\nu_{y,\eta} - \chi_{y,\eta}$ to have volume bounded independent of $y$. To overcome this obstacle, we use the structure \eqref{eqn:A,B_structure} of $A$ to get further control on $\nu_{y,\eta} - \chi_{y,\eta}$ in the region \eqref{eqn:nu-chi_portrait_end_trivial}, and thereby reduce the essential phase space support of $\nu_{y,\eta}-\chi_{y,\eta}$ to a much smaller subregion.
	
	This extra control we need is at frequencies $\sim \delta^{-2} \langle y \rangle$, which is a constant factor larger than what the uncertainty principle tells us we can handle. There is a simple model situation where one can understand the Fourier transform of a discrete measure beyond the uncertainty threshold: this is when the measure is a sum of Dirac masses on a lattice. Then by Poisson summation, the Fourier transform is a sum of Dirac masses on the dual lattice (up to a multiplicative constant). In particular, when the lattice is $\mathbf{Z} \subseteq \mathbf{R}$,
	\begin{align*}
		\widehat{\sum_{n \in \mathbf{Z}} \delta_n}
		= \sum_{n \in \mathbf{Z}} \delta_n.
	\end{align*}
	Since $\mathbf{Z}$ is $1$-separated, the uncertainty principle suggests that it should be hard to control this Fourier transform at frequencies $|\xi| \gtrsim 1$. However, the periodicity of $\mathbf{Z}$ leads to cancellation at all non-integer frequencies $\xi$. We will win by exploiting cancellation of this sort.
	
	Near the point $y \in \mathbf{R}^d$, it follows from \eqref{eqn:A,B_structure} and Taylor expansion that $A$ looks like a lattice $L$ where the distance between any two adjacent lattice points is $\sim \delta^2 \langle y \rangle^{-1}$.
	The measure $\nu_{y,\eta}$ will be a linear combination of Dirac masses on $A$, where the coefficients oscillate at frequency $\eta$.
	We therefore expect $\widehat{\nu_{y,\eta}}$ to be concentrated near $L^{\vee}+\eta$, where $L^{\vee}$ is the lattice dual to $L$. Indeed, we will find that the essential phase space support of $\nu_{y,\eta}-\chi_{y,\eta}$, intersected with $\{|\xi| \lesssim \delta^{-2} \langle y \rangle\}$, is contained in
	\begin{align}\label{eqn:nu-chi_portrait_sharp}
		\bigcup_{\substack{0 \neq \lambda \in L^{\vee} \\ |\lambda + \eta| \lesssim \delta^{-2}\langle y \rangle}} \{(x,\xi) \in \mathbf{R}^d \times \mathbf{R}^d : |x-y| \lesssim 1 \text{ and } |\xi-\eta-\lambda| \lesssim \delta^{-2}\}
	\end{align}
	(this follows from \eqref{eqn:end_portrait_main}).
	Here $\lambda = 0$ is excluded because of \eqref{eqn:nu-chi_portrait_end_trivial}.
	The distance between adjacent points in $L^{\vee}$ is $\sim \delta^{-2} \langle y \rangle$, so only $O(1)$ many $\lambda$'s appear in the union. Since $\langle \eta \rangle \leq \langle y \rangle$, we have $|\lambda+\eta| \sim |\lambda| \sim \delta^{-2}\langle y \rangle$ for each $\lambda$ which appears. Therefore \eqref{eqn:nu-chi_portrait_sharp} is a subregion of \eqref{eqn:nu-chi_portrait_end_trivial}, and \eqref{eqn:nu-chi_portrait_sharp} has volume $O(\delta^{-2d})$. As desired, this volume bound is independent of $y$.
	
	In the general case where $p$ and $q$ are not necessarily the same, there is one additional technical wrinkle needed to avoid losing powers of $\langle y \rangle$ in our estimates. So far, we have been using the phase space basis $(\chi_{y,\eta})_{y,\eta \in \mathbf{R}^d}$, with $\chi_{y,\eta}$ a wave packet adapted to $B_1(y) \times B_1(\eta)$. Here the choice of radius $1$ is arbitrary. Instead, we need to use a phase space basis $(\rho_{y,\eta})_{y,\eta \in \mathbf{R}^d}$, where $\rho_{y,\eta}$ is a wave packet adapted to $B_{R_{y,\eta}}(y) \times B_{R_{y,\eta}^{-1}}(\eta)$ in phase space, and the radius $R_{y,\eta}$ may vary with $y,\eta$. By coincidence, $R_{y,\eta} = 1$ works when $p = q = 1$, but it does not work in general (see Remark \ref{rem:R_y,eta}). We construct the basis $(\rho_{y,\eta})_{y,\eta \in \mathbf{R}^d}$ and describe the associated decomposition in this basis in Section \ref{subsec:modified_Gabor}.
	
	\section*{Acknowledgements}
	
	I thank Noah Kravitz and Anna Skorobogatova for helpful discussions, and Noah Kravitz and Maryna Viazovska for their feedback on earlier drafts. I also thank Mikhail Sodin for sharing an up-to-date version of the paper \cite{KNS}. Finally, I am grateful to my advisor, Peter Sarnak, for his enthusiasm and encouragement throughout this project. I am supported by the National Science Foundation Graduate Research Fellowship Program under Grant No. DGE-2039656.
	
	\section{Subcritical case: proof of Theorem \ref{thm:sub_non-end}} \label{sec:sub_non-end}
	
	\subsection{Surjectivity of $\mathcal{F}_{A,B}$} \label{subsec:sub_non-end_surjective} Assume in this section that $\delta$ is large. Then we wish to show that $\mathcal{F}_{A,B}$ is surjective. We follow the plan outlined in Section \ref{subsubsec:sub_non-end}. By translating $A,B$, assume without loss of generality that all points in $A,B$ have magnitude $\gtrsim \delta^{\varpi}$, where $\varpi \gtrsim 1$ is fixed. In particular, since $\delta$ is large, all points have magnitude $\geq 1$.
	
	For $a_0 \in A$, let $\varphi_{a_0}$ be a bump function adapted to
	\begin{align*}
		\{|x-a_0| \lesssim |a_0|^{-p}\},
	\end{align*}
	with $\varphi_{a_0}(a_0) = 1$. For $b_0 \in B$, let $\psi_{b_0}$ be such that $\widehat{\psi_{b_0}}$ is a bump function adapted to
	\begin{align*}
		\{|\xi-b_0| \lesssim |b_0|^{-q}\},
	\end{align*}
	with $\widehat{\psi_{b_0}}(b_0) = 1$. We choose these to be compactly supported as a matter of convenience, so that the $\approx$'s in Section \ref{subsubsec:sub_non-end} all become $=$'s, but this isn't really necessary. By the separation assumptions on $A$ and $B$,
	\begin{align}\label{eqn:phi_psi_indicator}
		\varphi_{a_0}|_A = 1_{\cdot = a_0}
		\qquad \text{and} \qquad
		\widehat{\psi_{b_0}}|_B = 1_{\cdot = b_0}.
	\end{align}
	Let $P \colon \mathcal{S}(A) \oplus \mathcal{S}(B) \to \mathcal{S}(\mathbf{R}^d)$ be given by
	\begin{align*}
		P(\alpha,\beta)
		= \sum_{a_0 \in A} \alpha(a_0) \varphi_{a_0} + \sum_{b_0 \in B} \beta(b_0) \psi_{b_0}.
	\end{align*}
	We want to show that the series
	\begin{align} \label{eqn:sub_non-end_series}
		\sum_{j=0}^{\infty} (1-\mathcal{F}_{A,B}P)^j
	\end{align}
	converges strongly in the space of continuous linear operators on $\mathcal{S}(A) \oplus \mathcal{S}(B)$. We will do this by proving an operator norm estimate of the form
	\begin{align} \label{eqn:1-FP_unif}
		\|(1-\mathcal{F}_{A,B}P)(\alpha,\beta)\| \lesssim o_{\delta \to \infty}(1) \|(\alpha,\beta)\|
	\end{align}
	for a countable family of norms which generate the Fr\'echet topology on $\mathcal{S}(A) \oplus \mathcal{S}(B)$. Here it is crucial that the constant and decay rate $o_{\delta \to \infty}(1)$ are uniform for all norms in the family. Then for $\delta$ sufficiently large,
	\begin{align*}
		\|(1-\mathcal{F}_{A,B}P)(\alpha,\beta)\|
		\leq \frac{1}{2} \|(\alpha,\beta)\|
	\end{align*}
	for all norms, and the series \eqref{eqn:sub_non-end_series} converges as desired. This uniformity fails for the most obvious family of norms, so we need to be careful in which norms we use. Our next step is to describe these norms.

	Given a weight function $w \colon [1,\infty) \to \mathbf{R}^+$ and an exponent $r > 0$, define
	\begin{align*}
		\|\alpha\|_{\mathcal{S}_w^r(A)}
		= \sum_{a \in A} |\alpha(a)| w(|a|^r)
	\end{align*}
	for $\alpha \in \mathcal{S}(A)$, and define $\|\beta\|_{\mathcal{S}_w^r(B)}$ for $\beta \in \mathcal{S}(B)$ analogously. These are weighted $\ell^1$ norms. Recall $|a|,|b| \geq 1$ for $a \in A$ and $b \in B$, so $w(|a|^r)$ and $w(|b|^r)$ are defined.
	
	Since we are in the subcritical case $pq < 1$, we can choose $\tilde{p} > p$ and $\tilde{q} > q$ with $\tilde{p} \tilde{q} = 1$ (on a first reading, we suggest that the reader focus on the case $p,q < 1$ and $\tilde{p} = \tilde{q} = 1$). Fix some $s_0 > 0$ large depending on $\tilde{p},\tilde{q}$, to be chosen later. \emph{For the remainder of Section \ref{sec:sub_non-end}, all constants will be allowed to depend implicitly on $\tilde{p},\tilde{q},s_0$, in addition to $p,q,d$.} For $s > s_0$, let $w_s$ be the weight function
	\begin{align*}
		w_s(t) = \max\{t^{s_0}, c_s t^s\},
	\end{align*}
	where $c_s > 0$ is a constant which is small depending on $s$. The freedom to choose $c_s$ small allows us to ensure that $w_s(t)$ will not start growing like $t^s$ until $t$ is as large as we like. On the other hand, $w_s(t)$ at least grows like $t^{s_0}$, so
	\begin{align}\label{eqn:fast_growth}
		w_s(t^{1+\gamma}) \geq t^{s_0\gamma} w_s(t)
	\end{align}
	for any $\gamma > 0$. For $(\alpha,\beta) \in \mathcal{S}(A) \oplus \mathcal{S}(B)$, let
	\begin{align*}
		\|(\alpha,\beta)\|_{\mathcal{S}_{w_s}^{\tilde{p}^{1/2}}(A) \oplus \mathcal{S}_{w_s}^{\tilde{q}^{1/2}}(B)}
		= \|\alpha\|_{\mathcal{S}_{w_s}^{\tilde{p}^{1/2}}(A)}
		+ \|\beta\|_{\mathcal{S}_{w_s}^{\tilde{q}^{1/2}}(B)}
	\end{align*}
	(although this looks complicated, it simplifies in the model case $\tilde{p} = \tilde{q} = 1$ because the exponents $\tilde{p},\tilde{q}$ don't affect the definition of the norm).
	This family of norms on $\mathcal{S}(A) \oplus \mathcal{S}(B)$, ranging over $s > s_0$, generates the topology on $\mathcal{S}(A) \oplus \mathcal{S}(B)$, and along this family we will have uniformity in \eqref{eqn:1-FP_unif}. The following proposition is the key estimate.
	
	\begin{prop}\label{prop:phi_psi_small}
		One has the bounds
		\begin{align*}
			\|\widehat{\varphi_{a_0}}|_B\|_{\mathcal{S}_{w_s}^{\tilde{q}^{1/2}}(B)}
			\lesssim \delta^{-1} w_s(|a_0|^{\tilde{p}^{1/2}})
			\qquad \text{and} \qquad
			\|\psi_{b_0}|_A\|_{\mathcal{S}_{w_s}^{\tilde{p}^{1/2}}(A)}
			\lesssim \delta^{-1} w_s(|b_0|^{\tilde{q}^{1/2}})
		\end{align*}
		uniformly in $s > s_0$.
	\end{prop}
	
	The factor $\delta^{-1}$ is not sharp, but all we need is that it goes to zero at a rate independent of $s$ as $\delta \to \infty$.
	
	\begin{proof}
		We only prove the first inequality, as the second is completely analogous. Write
		\begin{align*}
			\varphi_{a_0}(x) = \varphi(|a_0|^p(x-a_0))
		\end{align*}
		with $\varphi$ a standard bump function. Then
		\begin{align*}
			\|\widehat{\varphi_{a_0}}|_B\|_{\mathcal{S}_{w_s}^{\tilde{q}^{1/2}}}
			= \sum_{b \in B} |\widehat{\varphi_{a_0}}(b)| w_s(|b|^{\tilde{q}^{1/2}})
			= |a_0|^{-pd} \sum_{b \in B} |\widehat{\varphi}(|a_0|^{-p}b)| w_s(|b|^{\tilde{q}^{1/2}}).
		\end{align*}
		Let $\varepsilon > 0$ be such that $\tilde{p} = p+2\varepsilon$. We drop the factor of $|a_0|^{-pd} \leq 1$, and split the sum based on whether $|b|$ is smaller or bigger than $|a_0|^{p+\varepsilon}$. When $|b|$ is bigger, the Schwartz decay of $\widehat{\varphi}$ kicks in. Therefore
		\begin{align*}
			\|\widehat{\varphi_{a_0}}|_B\|_{\mathcal{S}_{w_s}^{\tilde{q}^{1/2}}}
			\lesssim \underbrace{\sum_{|b| \leq |a_0|^{p+\varepsilon}} w_s(|b|^{\tilde{q}^{1/2}})}_{\mbox{\RomanNumeralCaps{1}}}
			+ \underbrace{\sum_{|b| > |a_0|^{p+\varepsilon}} \min\{(|a_0|^{-p} |b|)^{-s_0'}, O_s(|a_0|^{-p} |b|)^{-s'}\} w_s(|b|^{\tilde{q}^{1/2}})}_{\mbox{\RomanNumeralCaps{2}}},
		\end{align*}
		where we choose $s_0',s'$ large depending on $s_0,s$, respectively.
		
		For $|b| \leq |a_0|^{p+\varepsilon}$,
		\begin{align*}
			|b|^{\tilde{q}^{1/2}}
			\leq |a_0|^{(p+\varepsilon)\tilde{q}^{1/2}}
			\qquad \text{and} \qquad
			(p+\varepsilon) \tilde{q}^{1/2}
			< \tilde{p}\tilde{q}^{1/2}
			= \tilde{p}^{1/2},
		\end{align*}
		so by \eqref{eqn:fast_growth} and the fact that there are only polynomially many $b$'s in the sum,
		\begin{align*}
			\mbox{\RomanNumeralCaps{1}}
			\lesssim |a_0|^{-(s_1 - O(1))} w_s(|a_0|^{\tilde{p}^{1/2}})
			\qquad \text{for some} \qquad
			s_1 \gtrsim s_0.
		\end{align*}
		Since $s_0$ is large, the exponent $s_1 - O(1)$ can be taken to be at least $1/\varpi$. Combined with $|a_0| \gtrsim \delta^{\varpi}$ (which we arranged at the beginning of this section), we obtain
		\begin{align*}
			\mbox{\RomanNumeralCaps{1}}
			\lesssim \delta^{-1} w_s(|a_0|^{\tilde{p}^{1/2}}).
		\end{align*}
		
		The summand in \RomanNumeralCaps{2} is a product of a min and a max (because $w_s$ is a max of two quantities). Estimating this product in two different ways based on which argument of the max is bigger,
		\begin{align*}
			\mbox{\RomanNumeralCaps{2}}
			\lesssim \sum_{|b| > |a_0|^{p+\varepsilon}} (|a_0|^{-p}|b|)^{-s_0'} |b|^{\tilde{q}^{1/2}s_0} + \sum_{|b| > |a_0|^{p+\varepsilon}} O_s(|a_0|^{-p} |b|)^{-s'} c_s |b|^{\tilde{q}^{1/2}s}.
		\end{align*}
		Taking $s_0',s'$ large enough, and taking $c_s$ small enough to cancel out the implied constant, a dyadic decomposition in $|b|$ easily yields
		\begin{align*}
			\mbox{\RomanNumeralCaps{2}}
			\lesssim |a_0|^{-1/\varpi}
			\lesssim \delta^{-1}
			\leq \delta^{-1} w_s(|a_0|^{\tilde{p}^{1/2}}).
		\end{align*}
		Hence
		\begin{align*}
			\|\widehat{\varphi_{a_0}}|_B\|_{\mathcal{S}_{w_s}^{\tilde{q}^{1/2}}}
			&\lesssim \mbox{\RomanNumeralCaps{1}} + \mbox{\RomanNumeralCaps{2}}
			\lesssim \delta^{-1} w_s(|a_0|^{\tilde{p}^{1/2}}).
			\qedhere
		\end{align*}
	\end{proof}
	
	Now fix $\delta$ to be at least $2$ times the implied constants in Proposition \ref{prop:phi_psi_small}. Then by \eqref{eqn:phi_psi_indicator} and the proposition,
	\begin{align*}
		\|(1-\mathcal{F}_{A,B} P) (1_{\cdot = a_0},0)\|_{\mathcal{S}_{w_s}^{\tilde{p}^{1/2}}(A) \oplus \mathcal{S}_{w_s}^{\tilde{q}^{1/2}}(B)}
		\leq \frac{1}{2} \|(1_{\cdot = a_0}, 0)\|_{\mathcal{S}_{w_s}^{\tilde{p}^{1/2}}(A) \oplus \mathcal{S}_{w_s}^{\tilde{q}^{1/2}}(B)}
	\end{align*}
	and similarly
	\begin{align*}
		\|(1 - \mathcal{F}_{A,B} P) (0,1_{\cdot = b_0})\|_{\mathcal{S}_{w_s}^{\tilde{p}^{1/2}}(A) \oplus \mathcal{S}_{w_s}^{\tilde{q}^{1/2}}(B)}
		\leq \frac{1}{2} \|(0,1_{\cdot=b_0})\|_{\mathcal{S}_{w_s}^{\tilde{p}^{1/2}}(A) \oplus \mathcal{S}_{w_s}^{\tilde{q}^{1/2}}(B)}.
	\end{align*}
	Thus by the definition of the norm and the triangle inequality,
	\begin{align*}
		\|(1 - \mathcal{F}_{A,B} P) (\alpha,\beta)\|_{\mathcal{S}_{w_s}^{\tilde{p}^{1/2}}(A) \oplus \mathcal{S}_{w_s}^{\tilde{q}^{1/2}}(B)}
		\leq \frac{1}{2} \|(\alpha,\beta)\|_{\mathcal{S}_{w_s}^{\tilde{p}^{1/2}}(A) \oplus \mathcal{S}_{w_s}^{\tilde{q}^{1/2}}(B)}
	\end{align*}
	for all $(\alpha,\beta) \in \mathcal{S}(A) \oplus \mathcal{S}(B)$. Hence the power series
	\begin{align*}
		\sum_{j=0}^{\infty} (1-\mathcal{F}_{A,B}P)^j
	\end{align*}
	converges in operator norm with respect to the norm associated to $s$, for all $s > s_0$. In particular, it converges strongly, and defines an operator on $\mathcal{S}(A) \oplus \mathcal{S}(B)$ which inverts $\mathcal{F}_{A,B}P$. It follows that $\mathcal{F}_{A,B}$ is surjective, as desired.
	
	\subsection{Completing the proof of Theorem \ref{thm:sub_non-end}} \label{subsec:completing_sub_non-end} Now suppose $\delta$ is arbitrary. We wish to show that $\Coker \mathcal{F}_{A,B}$ has dimension $O_{\delta}(1)$, and $\Ker \mathcal{F}_{A,B}$ is infinite-dimensional.
	
	To prove the former, remove $O_{\delta}(1)$ points from $A,B$ to form subsets $A' \subseteq A$ and $B' \subseteq B$ which are $(p',\delta')$-separated and $(q',\delta')$-separated, respectively, for fixed $p' > p$ and $q' > q$ with $p'q' < 1$, and for $\delta'$ sufficiently large depending on $p',q'$. Then $\mathcal{F}_{A',B'}$ is surjective by the argument above, so it follows immediately that $\dim \Coker \mathcal{F}_{A,B} = O_{\delta}(1)$.
	
	To prove the latter, add infinitely many points to $A,B$ to form supersets $A' \supseteq A$ and $B' \supseteq B$ which are $(p,\delta')$-separated and $(q,\delta')$-separated, respectively, for some $\delta' \gtrsim \delta$. The subspace of $\mathcal{S}(A') \oplus \mathcal{S}(B')$ consisting of sequences which vanish on $A$ and $B$ is infinite-dimensional. Since $\Coker \mathcal{F}_{A',B'}$ is finite-dimensional, this subspace must meet $\im \mathcal{F}_{\mathcal{A}',\mathcal{B}'}$ in an infinite-dimensional subspace $\mathcal{E}$. The preimage of $\mathcal{E}$ under $\mathcal{F}_{A',B'}$ is an infinite-dimensional subspace of $\Ker \mathcal{F}_{A,B}$.
	
	\section{Sparse critical case: proof of Theorem \ref{thm:sub_end}} \label{sec:sub_end}
	
	
	We follow the outline in Section \ref{subsubsec:sub_end}. By the same idea as in the last paragraph of Section \ref{subsec:completing_sub_non-end}, it suffices to show that the image of $\mathcal{F}_{A,B}$ contains all finitely supported sequences.
	
	As in Section \ref{subsec:sub_non-end_surjective}, we may assume all points in $A,B$ have magnitude $\gtrsim \delta^{\varpi}$ (with $\varpi \gtrsim 1$), in particular $\geq 1$ because $\delta$ is large.
	
	Let $\rho \in \mathcal{S}(\mathbf{R}^d)$ be such that $\hat{\rho}$ is a standard bump function and $\rho(0) = 1$. Furthermore, take $\hat{\rho}$ to be Gevrey, so $\rho$ satisfies the quasi-exponential decay condition
	\begin{align}\label{eqn:quasi-exp_decay}
		|\rho(x)| \lesssim e^{-|x|^{0.9}}.
	\end{align}
	Set
	\begin{align}\label{eqn:phi_psi_Gevrey}
		\varphi_{a_0} = \rho(\delta^{-1/2} |a_0|^p (\cdot-a_0))
		\qquad \text{and} \qquad
		\psi_{b_0} = \rho(\delta^{-1/2} |b_0|^q (\cdot - b_0))^{\vee}.
	\end{align}
	Define $P \colon \mathcal{S}(A) \oplus \mathcal{S}(B) \to \mathcal{S}(\mathbf{R}^d)$ by
	\begin{align*}
		P(\alpha,\beta)
		= \sum_{a_0 \in A} \alpha(a_0) \varphi_{a_0} + \sum_{b_0 \in B} \beta(b_0) \psi_{b_0}
	\end{align*}
	as in the subcritical case.
	
	Let $w \colon [1,\infty) \to \mathbf{R}^+$ be the weight function
	\begin{align*}
		w(t) = e^{\log^3 t}.
	\end{align*}
	There's a lot of freedom in how to choose $w$ --- it just has to grow faster than $e^{C\log^2 t}$ and slower than $e^{t^c}$ for $C$ a large constant and $c$ a small constant. In analogy with \eqref{eqn:fast_growth}, we will use that $w$ grows fast enough that
	\begin{align}\label{eqn:super_fast_growth}
		w(2t) \gtrsim_j t^j w(t)
	\end{align}
	for any $j$. On the other hand, $w$ grows much slower than the decay rate \eqref{eqn:quasi-exp_decay} of $\rho$.
	
	We will show that all sequences in $\mathcal{S}_w^{p^{1/2}}(A) \oplus \mathcal{S}_w^{q^{1/2}}(B)$ lie in the image of $\mathcal{F}_{A,B}$, where $\mathcal{S}_w^{p^{1/2}}(A) \oplus \mathcal{S}_w^{q^{1/2}}(B)$ is the space of $(\alpha,\beta) \in \mathcal{S}(A) \oplus \mathcal{S}(B)$ for which the norm
	\begin{align}\label{eqn:quasi-exp_norm}
		\|(\alpha,\beta)\|_{\mathcal{S}_w^{p^{1/2}}(A) \oplus \mathcal{S}_w^{q^{1/2}}(B)}
		= \sum_{a \in A} |\alpha(a)| w(|a|^{p^{1/2}}) + \sum_{b \in B} |\beta(b)| w(|b|^{q^{1/2}})
	\end{align}
	is finite. This space is what we called $\mathcal{T}(A,B)$ in Section \ref{subsubsec:sub_end}. Of course, this contains all finitely supported sequences $(\alpha,\beta)$.
	
	To show $\mathcal{S}_w^{p^{1/2}}(A) \oplus \mathcal{S}_w^{q^{1/2}}(B) \subseteq \im \mathcal{F}_{A,B}$, it suffices as in Section \ref{subsec:sub_non-end_surjective} to establish the operator norm bound
	\begin{align*}
		\|(1 - \mathcal{F}_{A,B}P) (\alpha,\beta)\|_{\mathcal{S}_w^{p^{1/2}}(A) \oplus \mathcal{S}_w^{q^{1/2}}(B)}
		\leq \frac{1}{2} \|(\alpha,\beta)\|_{\mathcal{S}_w^{p^{1/2}}(A) \oplus \mathcal{S}_w^{q^{1/2}}(B)}.
	\end{align*}
	By the triangle inequality and the definition of the norm \eqref{eqn:quasi-exp_norm}, this bound follows from the estimates
	\begin{gather}
		\|\varphi_{a_0}|_A - 1_{\cdot = a_0}\|_{\mathcal{S}_w^{p^{1/2}}}
		\leq \frac{1}{4} w(|a_0|^{p^{1/2}}),
		\label{eqn:phi_phys_approx}
		\\
		\|\widehat{\varphi_{a_0}}|_B\|_{\mathcal{S}_w^{q^{1/2}}}
		\leq \frac{1}{4} w(|a_0|^{p^{1/2}}),
		\label{eqn:phi_fourier_approx}
		\\
		\|\psi_{b_0}|_A\|_{\mathcal{S}_w^{p^{1/2}}}
		\leq \frac{1}{4} w(|b_0|^{q^{1/2}}),
		\label{eqn:psi_phys_approx}
		\\
		\|\widehat{\psi_{b_0}}|_B - 1_{\cdot = b_0}\|_{\mathcal{S}_w^{q^{1/2}}}
		\leq \frac{1}{4} w(|b_0|^{q^{1/2}}),
		\label{eqn:psi_fourier_approx}
	\end{gather}
	for all $a_0 \in A$ and $b_0 \in B$. So we have reduced the proof of Theorem \ref{thm:sub_end} to the four estimates \eqref{eqn:phi_phys_approx}, \eqref{eqn:phi_fourier_approx}, \eqref{eqn:psi_phys_approx}, and \eqref{eqn:psi_fourier_approx}.
	
	We begin with \eqref{eqn:phi_phys_approx}. By definition,
	\begin{align*}
		\|\varphi_{a_0}|_A - 1_{\cdot = a_0}\|_{\mathcal{S}_w^{p^{1/2}}}
		= \sum_{a \neq a_0} |\rho(\delta^{-1/2} |a_0|^p (a-a_0))| w(|a|^{p^{1/2}}).
	\end{align*}
	The fast decay \eqref{eqn:quasi-exp_decay} of $\rho$ relative to the much slower growth of $w$, combined with the $(p,\delta)$-separation of $A$, gives us that this sum is dominated by the terms where $a$ is as close as possible to $a_0$, i.e., where $|a-a_0| \sim \delta |a_0|^{-p}$. Therefore
	\begin{align*}
		\|\varphi_{a_0}|_A - 1_{\cdot = a_0}\|_{\mathcal{S}_w^{p^{1/2}}}
		\lesssim e^{-(\delta^{1/2})^{0.9}} w((|a_0|+O(\delta|a_0|^{-p}))^{p^{1/2}}).
	\end{align*}
	If $|a_0|$ is bigger than a large power of $\delta$, then
	\begin{align*}
		w((|a_0|+O(\delta|a_0|^{-p}))^{p^{1/2}})
		\leq w((|a_0|+1)^{p^{1/2}})
		\sim w(|a_0|^{p^{1/2}}),
	\end{align*}
	while if $|a_0| \lesssim \delta^{O(1)}$, then
	\begin{align*}
		w((|a_0|+O(\delta|a_0|^{-p}))^{p^{1/2}})
		\leq w(\delta^{O(1)})
		\leq w(\delta^{O(1)}) w(|a_0|^{p^{1/2}}).
	\end{align*}
	In both cases, it follows that
	\begin{align*}
		\|\varphi_{a_0}|_A - 1_{\cdot = a_0}\|_{\mathcal{S}_w^{p^{1/2}}}
		\lesssim e^{-(\delta^{1/2})^{0.8}} w(|a_0|^{p^{1/2}}).
	\end{align*}
	Taking $\delta$ sufficiently large (independent of $a_0$), we obtain \eqref{eqn:phi_phys_approx}. The same argument proves \eqref{eqn:psi_fourier_approx}.
	
	For \eqref{eqn:phi_fourier_approx}, write
	\begin{align*}
		|\widehat{\varphi_{a_0}}(\xi)|
		= \delta^{d/2} |a_0|^{-pd} |\hat{\rho}(\delta^{1/2} |a_0|^{-p} \xi)|.
	\end{align*}
	Since $\hat{\rho}$ has compact support, $\widehat{\varphi_{a_0}}$ is supported on $\{|\xi| \lesssim \delta^{-1/2} |a_0|^p\}$. Therefore
	\begin{align*}
		\|\widehat{\varphi_{a_0}}|_B\|_{\mathcal{S}_w^{q^{1/2}}}
		= \sum_{b \in B} |\widehat{\varphi_{a_0}}(b)| w(|b|^{q^{1/2}})
		\lesssim |a_0|^{O(1)} w(O(\delta^{-1/2}|a_0|^p)^{q^{1/2}})
		= |a_0|^{O(1)} w(O(\delta^{-q^{1/2}/2} |a_0|^{p^{1/2}})),
	\end{align*}
	where in the last equality we have used that $pq = 1$. Taking $\delta$ sufficiently large, we get
	\begin{align*}
		\|\widehat{\varphi_{a_0}}|_B\|_{\mathcal{S}_w^{q^{1/2}}}
		\lesssim |a_0|^{O(1)} w\Big(\frac{1}{2} |a_0|^{p^{1/2}}\Big).
	\end{align*}
	Using \eqref{eqn:super_fast_growth}, we can control this by
	\begin{align*}
		\|\widehat{\varphi_{a_0}}|_B\|_{\mathcal{S}_w^{q^{1/2}}}
		\lesssim |a_0|^{-1} w(|a_0|^{p^{1/2}}).
	\end{align*}
	Since $|a_0| \gtrsim \delta^{\varpi}$, we obtain \eqref{eqn:phi_fourier_approx} when $\delta$ is large. Once again, the same argument gives the analogous estimate \eqref{eqn:psi_phys_approx} for $\psi$.
	
	We have now established \eqref{eqn:phi_phys_approx}, \eqref{eqn:phi_fourier_approx}, \eqref{eqn:psi_phys_approx}, \eqref{eqn:psi_fourier_approx}, so the proof of Theorem \ref{thm:sub_end} is complete.
	
	\section{Supercritical case: proof of Theorem \ref{thm:super_non-end}} \label{sec:super_non-end}
	
	Sections \ref{subsec:sampling} and \ref{subsec:Gabor} prepare us for Section \ref{subsec:injectivity_non-end}, where we implement the strategy in Section \ref{subsubsec:super_non-end}. The proof is finished off in Section \ref{subsec:proof_super_non-end}, which is quite similar to Section \ref{subsec:completing_sub_non-end} in the subcritical case.

	\subsection{Integration by sampling} \label{subsec:sampling}
	
	We are concerned with the following natural question: how best to estimate the integral of a function, knowing only its values on a given discrete sample set? Lemmas \ref{lem:local_mu} and \ref{lem:full_mu} provide partial answers.
	
	In this section alone, we include in our notation the dependence of constants on $p,d$.
	
	\begin{lem}\label{lem:local_mu}
		Fix $k \in \mathbf{Z}^+$, and let $\varepsilon > 0$ be sufficiently small depending on $k$ and $d$. Suppose $E$ is an $\varepsilon r$-dense subset of a cube $Q \subseteq \mathbf{R}^d$ of sidelength $r$. Then there is a signed measure $\mu$ supported on $E$ of total variation $\|\mu\| \lesssim_{k,d} r^d$, such that
		\begin{align}\label{eqn:local_approx}
			\int_Q \varphi = \int \varphi \d\mu + O_{k,d}(r^{d+k} \|\nabla^k \varphi\|_{\infty})
		\end{align}
		for all $\varphi \in C^k(Q)$.
	\end{lem}
	
	\begin{proof}
		By translation invariance and rescaling, we may assume $Q = [0,1]^d$ and $r = 1$. We approximate $\varphi$ by its $(k-1)$-jet at $0$:
		\begin{align}\label{eqn:taylor}
			\varphi(x) = J_0^{(k-1)}\varphi(x) + O_{k,d}(\|\nabla^k\varphi\|_{\infty}),
			\qquad \text{where} \qquad
			J_0^{(k-1)}\varphi(x)
			= \sum_{|\alpha| \leq k-1} \frac{1}{\alpha !} \partial^{\alpha} \varphi(0) x^{\alpha}.
		\end{align}
		This jet is an element of $\mathcal{P}$, the vector space of polynomials of degree at most $k-1$ in $d$ variables. Denote $N = \dim \mathcal{P} \lesssim_{k,d} 1$. Fix $N$ points in $[0,1]^d$ such that only the zero polynomial in $\mathcal{P}$ vanishes at all $N$ points. Assuming $\varepsilon$ is sufficiently small depending on $k$ and $d$, we can approximate these $N$ points by $x_1,\dots,x_N \in E$, with the approximation good enough that the bound
		\begin{align}\label{eqn:P_recovery}
			\|P\|_{\mathcal{P}} \lesssim_{k,d} \sum_{j=1}^N |P(x_j)|
		\end{align}
		holds for all $P \in \mathcal{P}$ (here $\|\cdot\|_{\mathcal{P}}$ is some fixed choice of norm on $\mathcal{P}$). Let $\mu$ be the unique measure supported on $\{x_1,\dots,x_N\}$ such that
		\begin{align}\label{eqn:mu_def}
			\int_{[0,1]^d} P = \int P \d\mu
		\end{align}
		for all $P \in \mathcal{P}$. It follows from \eqref{eqn:P_recovery} that $\mu$ has total variation $O_{k,d}(1)$. Integrating the Taylor approximation \eqref{eqn:taylor} and applying \eqref{eqn:mu_def}, we conclude that
		\begin{gather*}
			\int_{[0,1]^d} \varphi
			= \int \varphi \d\mu + O_{k,d}(\|\nabla^k\varphi\|_{\infty}).
			\qedhere
		\end{gather*}
	\end{proof}
	
	Let $M_1^{\infty}$ be the ``unit scale $L^{\infty}$ maximal operator"
	\begin{align*}
		M_1^{\infty}f(x)
		= \|f\|_{L^{\infty}(B_1(x))}.
	\end{align*}
	
	\begin{lem}\label{lem:full_mu}
		Fix $k \in \mathbf{Z}^+$. Let $A \subseteq \mathbf{R}^d$ be $(p,\delta)$-dense with $p \geq 0$ and $\delta$ small depending on $k,p,d$. Then there is a signed measure $\mu$ supported on $A$, with total variation $O_{k,p,d}(1)$ on each unit ball in $\mathbf{R}^d$, such that
		\begin{align}\label{eqn:full_approx}
			\int_{\mathbf{R}^d} f = \int f \d\mu + O_{k,p,d}\Big(\delta^k \int_{\mathbf{R}^d} M_1^{\infty}[\nabla^k f](x) \langle x \rangle^{-pk} \d x \Big)
		\end{align}
		for all Schwartz functions $f \in \mathcal{S}(\mathbf{R}^d)$.
	\end{lem}
	
	This could be sharpened by allowing the scale $1$ in $M_1^{\infty}$ to decrease as $|x| \to \infty$, but that would only make a difference if the graph of $|\nabla^k f|$ had isolated spikes of width $\lll 1$, which doesn't occur in our application. Similarly, the total variation of $\mu$ can be controlled on smaller balls when the center of the ball is far from the origin, but this isn't necessary for us.
	
	\begin{proof}
		Partition $\mathbf{R}^d$ into dyadic cubes so that the cube $Q_x$ containing $x$ has sidelength $\sim_{p,d} C\delta\langle x \rangle^{-p}$ for all $x \in \mathbf{R}^d$, where $C \gg_{k,p,d} 1$ is a fixed constant. Then $A$ will be dense enough on each cube that Lemma \ref{lem:local_mu} applies. Let $\mu$ be the sum of the resulting measures on each cube. Then \eqref{eqn:full_approx} follows by summing \eqref{eqn:local_approx} over the cubes (here it's important for $\delta$ to be small enough that $Q_x \subseteq B_1(x)$ for all $x$, so $M_1^{\infty}[\nabla^kf](x)$ is an upper bound for $\|\nabla^k f\|_{L^{\infty}(Q_x)}$).
	\end{proof}
	
	\subsection{The Gabor transform} \label{subsec:Gabor}
	
	Fix a standard Schwartz function $\chi \in \mathcal{S}(\mathbf{R}^d)$ with $L^2$ norm equal to $1$. For $y,\eta \in \mathbf{R}^d$, denote
	\begin{align}\label{eqn:time_freq_shift}
		\chi_{y,\eta}(x) = \chi(x-y) e^{2\pi i\eta \cdot x},
	\end{align}
	so $\chi_{y,\eta}$ is a wave packet adapted to $B_1(y) \times B_1(\eta)$ in phase space. The \emph{Gabor transform} $Tf$ of a tempered distribution $f \in \mathcal{S}'(\mathbf{R}^d)$ is
	\begin{align*}
		Tf(y,\eta) = \langle f,\chi_{y,\eta} \rangle
		= \int_{\mathbf{R}^d} f(x) \overline{\chi_{y,\eta}(x)} \d x
	\end{align*}
	(typically $\chi$ is chosen to be a Gaussian, but all we need is that $\chi$ is Schwartz, so we prefer to leave it unspecified) \cite{Folland, Tao_pseudo}.
	Heuristically, this measures the extent to which $f|_{B_{O(1)}(y)}$ oscillates at frequencies $\eta+O(1)$. When $f$ is Schwartz, $Tf(y,\eta)$ decays faster than any polynomial in $y,\eta$. For $f$ a general tempered distribution, $Tf(y,\eta)$ grows at most polynomially in $y,\eta$. The \emph{Gabor inversion formula} recovers $f$ from $Tf$, expressing $f$ as an integral linear combination of the wave packets $\chi_{y,\eta}$. Such an expression is not unique, but it is essentially unique at scales $\gg 1$, because $\chi_{y,\eta}$ is almost orthogonal to $\chi_{y',\eta'}$ when $|y-y'| \gg 1$ or $|\eta-\eta'| \gg 1$.
	
	\begin{prop}[Gabor inversion formula] \label{prop:Gabor_inversion}
		For all Schwartz functions $f \in \mathcal{S}(\mathbf{R}^d)$,
		\begin{align*}
			f = \int_{\mathbf{R}^d \times \mathbf{R}^d} Tf(y,\eta) \chi_{y,\eta} \d y \d\eta.
		\end{align*}
		This holds more generally for tempered distributions when interpreted in the obvious way (i.e., as a Gelfand--Pettis integral).
	\end{prop}
	
	This is well-known, but the proof is only a few lines, so we include it for the reader's convenience.
	
	\begin{proof}
		It suffices to prove this for Schwartz functions, since it will then extend to tempered distributions by continuity. Integrating in $\eta$ first,
		\begin{align*}
			\int_{\mathbf{R}^d \times \mathbf{R}^d} Tf(y,\eta) \chi_{y,\eta}(x) \d y \d\eta
			= \int_{\mathbf{R}^d} \chi(x-y) Tf(y,\cdot)^{\vee}(x) \d y.
		\end{align*}
		We can write
		\begin{align*}
			Tf(y,\eta) = \widehat{\overline{\chi}(\cdot - y) f(\cdot)}(\eta).
		\end{align*}
		Thus by Fourier inversion,
		\begin{align*}
			\int_{\mathbf{R}^d \times \mathbf{R}^d} Tf(y,\eta) \chi_{y,\eta}(x) \d y \d\eta
			= \int_{\mathbf{R}^d} |\chi(x-y)|^2 f(x) \d y
			= f(x).
		\end{align*}
		The last equality is because $\|\chi\|_{L^2} = 1$.
	\end{proof}
	
	
	For $s \in \mathbf{R}$, let $\mathcal{S}^{-s}(\mathbf{R}^d)$ be the Banach space consisting of tempered distributions $f \in \mathcal{S}'(\mathbf{R}^d)$ for which the norm
	\begin{align*}
		\|f\|_{\mathcal{S}^{-s}}
		= \int_{\mathbf{R}^d \times \mathbf{R}^d} |Tf(y,\eta)| \langle y \rangle^{-s} \langle \eta \rangle^{-s} \d y \d\eta
	\end{align*}
	is finite. Schwartz functions are dense in $\mathcal{S}^{-s}(\mathbf{R}^d)$. The inclusions $\mathcal{S}(\mathbf{R}^d) \hookrightarrow \mathcal{S}^{-s}(\mathbf{R}^d) \hookrightarrow \mathcal{S}'(\mathbf{R}^d)$ are continuous. By integration by parts, one can easily check that
	\begin{align*}
		\|\chi_{y,\eta}\|_{\mathcal{S}^{-s}}
		\lesssim_s \langle y \rangle^{-s} \langle \eta \rangle^{-s}
		\qquad \text{and} \qquad
		\|\widehat{\chi_{y,\eta}}\|_{\mathcal{S}^{-s}}
		\lesssim_s \langle y \rangle^{-s} \langle \eta \rangle^{-s}.
	\end{align*}
	From the latter and the Gabor inversion formula, it follows that the Fourier transform acts as a bounded operator on $\mathcal{S}^{-s}(\mathbf{R}^d)$.
	
	
	\subsection{Injectivity of $\mathcal{F}_{A,B}$} \label{subsec:injectivity_non-end}
	
	We execute the argument outlined in Section \ref{subsubsec:super_non-end}.
	
	Fix $0 < \tilde{p} < p$ and $0 < \tilde{q} < q$ with $\tilde{p} \tilde{q} = 1$. On a first reading, we recommend that the reader takes $p,q > 1$ and $\tilde{p} = \tilde{q} = 1$. \emph{For the remainder of this section, all constants will be allowed to depend implicitly on $\tilde{p},\tilde{q}$, in addition to $p,q,d$.} Let $k$ be large, to be chosen later. Taking $\delta$ sufficiently small depending on $k$, there are measures $\mu_A,\mu_B$ supported on $A,B$ as in Lemma \ref{lem:full_mu}. For $y,\eta \in \mathbf{R}^d$, let
	\begin{align} \label{eqn:nu_def_non-end}
		\nu_{y,\eta}
		=
		\begin{cases}
			\chi_{y,\eta} \mu_A &\mbox{if } \langle y \rangle^{\tilde{p}^{1/2}}
			\geq \langle \eta \rangle^{\tilde{q}^{1/2}},
			\\
			(\widehat{\chi_{y,\eta}} \mu_B)^{\vee} &\mbox{if } \langle y \rangle^{\tilde{p}^{1/2}} < \langle \eta \rangle^{\tilde{q}^{1/2}}.
		\end{cases}
	\end{align}
	Then $\nu_{y,\eta}$ is either a tempered measure on $A$ or the inverse Fourier transform of a tempered measure on $B$, so $\nu_{y,\eta} \in \mathcal{M}_{A,B}$ (recall $\mathcal{M}_{A,B}$ was defined in Section \ref{subsubsec:super_non-end}).
	
	We remark that the inequalities $\langle y \rangle^{\tilde{p}^{1/2}} \geq \langle \eta \rangle^{\tilde{q}^{1/2}}$ and $\langle y \rangle^{\tilde{p}^{1/2}} < \langle \eta \rangle^{\tilde{q}^{1/2}}$ are written in the most symmetric way possible, but what we'll actually use are the equivalent inequalities
	\begin{align*}
		\langle y \rangle^{\tilde{p}^{1/2}}
		\geq \langle \eta \rangle^{\tilde{q}^{1/2}}
		\iff
		\langle \eta \rangle \leq \langle y \rangle^{\tilde{p}}
		\qquad \text{and} \qquad
		\langle y \rangle^{\tilde{p}^{1/2}}
		< \langle \eta \rangle^{\tilde{q}^{1/2}}
		\iff
		\langle y \rangle < \langle \eta \rangle^{\tilde{q}};
	\end{align*}
	these equivalences are where we need $\tilde{p}\tilde{q} = 1$.
	
	
	\begin{prop}\label{prop:nu_approx}
		Assume $s$ is sufficiently large, $k$ is sufficiently large depending on $s$, and $\delta$ is sufficiently small depending on $k,s$. Then
		\begin{align}\label{eqn:nu_approx}
			\|\nu_{y,\eta} - \chi_{y,\eta}\|_{\mathcal{S}^{-s}}
			\leq \frac{1}{2} \langle y \rangle^{-s} \langle \eta \rangle^{-s}.
		\end{align}
	\end{prop}
	
	\begin{proof}
		We first consider the case $\langle y \rangle^{\tilde{p}^{1/2}} \geq \langle \eta \rangle^{\tilde{q}^{1/2}}$, so $\nu_{y,\eta} = \chi_{y,\eta} \mu_A$. To estimate $\|\nu_{y,\eta} - \chi_{y,\eta}\|_{\mathcal{S}^{-s}}$, we want pointwise control on $T(\nu_{y,\eta} - \chi_{y,\eta})(z,\zeta)$. The trivial bound, using the triangle inequality and the fact that $\mu_A$ has total variation $O_k(1)$ on unit balls, is
		\begin{align}\label{eqn:T(error)_triv}
			|T(\nu_{y,\eta} - \chi_{y,\eta})(z,\zeta)|
			\leq |T\nu_{y,\eta}(z,\zeta)| + |T\chi_{y,\eta}(z,\zeta)|
			\lesssim_k \langle y-z \rangle^{-\infty}.
		\end{align}
		When $\zeta$ is not too large, we can do better ($\eta$ is not too large because $\langle \eta \rangle \leq \langle y \rangle^{\tilde{p}}$). By \eqref{eqn:full_approx},
		\begin{align*}
			|T(\nu_{y,\eta} - \chi_{y,\eta})(z,\zeta)|
			&= \Big|\int \chi_{y,\eta} \overline{\chi_{z,\zeta}} \d\mu_A - \int_{\mathbf{R}^d} \chi_{y,\eta} \overline{\chi_{z,\zeta}}\Big|
			\\&\lesssim_k \delta^k \int_{\mathbf{R}^d} M_1^{\infty}[\nabla^k(\chi_{y,\eta} \overline{\chi_{z,\zeta}})](x) \langle x \rangle^{-pk} \d x
			\\&\lesssim_k \delta^k \langle y-z \rangle^{-\infty} \max\{\langle \eta \rangle, \langle \zeta \rangle\}^k \langle y \rangle^{-pk}.
		\end{align*}
		Write $p = \tilde{p} + 2\varepsilon$, and suppose $\langle \zeta \rangle \lesssim \delta^{-1/2} \langle y \rangle^{\tilde{p}+\varepsilon}$ (this is our condition that $\zeta$ is not too large). Then
		\begin{align*}
			\max\{\langle \eta \rangle, \langle \zeta \rangle\}
			\lesssim \delta^{-1/2} \langle y \rangle^{\tilde{p}+\varepsilon}
			= \delta^{-1/2} \langle y \rangle^{p-\varepsilon}.
		\end{align*}
		We thus obtain
		\begin{align*}
			|T(\nu_{y,\eta} - \chi_{y,\eta})(z,\zeta)|
			\lesssim_k \delta^{k/2} \langle y-z \rangle^{-\infty} \langle y \rangle^{-k\varepsilon}.
		\end{align*}
		Combining this with the trivial bound \eqref{eqn:T(error)_triv} for large $\zeta$,
		\begin{align*}
			\|\nu_{y,\eta} - \chi_{y,\eta}\|_{\mathcal{S}^{-s}}
			&\lesssim_k \int_{\{(z,\zeta) : \langle \zeta \rangle \lesssim \delta^{-1/2} \langle y \rangle^{\tilde{p}+\varepsilon}\}} \delta^{k/2} \langle y-z \rangle^{-\infty} \langle y \rangle^{-k\varepsilon} \langle z \rangle^{-s} \langle \zeta \rangle^{-s} \d z \d\zeta
			\\&\qquad+ \int_{\{(z,\zeta) : \langle \zeta \rangle \gtrsim \delta^{-1/2} \langle y \rangle^{\tilde{p}+\varepsilon}\}} \langle y-z \rangle^{-\infty} \langle z \rangle^{-s} \langle \zeta \rangle^{-s} \d z \d\zeta
			\\&\lesssim_{k,s} \delta^{k/2} \langle y \rangle^{-s-k\varepsilon} + \langle y \rangle^{-s} (\delta^{-1/2} \langle y \rangle^{\tilde{p}+\varepsilon})^{-(s-d)}.
		\end{align*}
		Since $\langle \eta \rangle \leq \langle y \rangle^{\tilde{p}}$, we can estimate the right hand side by
		\begin{align*}
			\|\nu_{y,\eta} - \chi_{y,\eta}\|_{\mathcal{S}^{-s}}
			\lesssim_{k,s} [\delta^{k/2} \langle y \rangle^{\tilde{p}s-k\varepsilon} + \delta^{(s-d)/2} \langle y \rangle^{\tilde{p}d - \varepsilon(s-d)}] \langle y \rangle^{-s} \langle \eta \rangle^{-s}.
		\end{align*}
		Taking $s$ sufficiently large, $k$ sufficiently large depending on $s$, and $\delta$ sufficiently small depending on $k,s$, we obtain the desired bound \eqref{eqn:nu_approx} in the case $\langle y \rangle^{\tilde{p}^{1/2}} \geq \langle \eta \rangle^{\tilde{q}^{1/2}}$.
		
		In the other case $\langle y \rangle^{\tilde{p}^{1/2}} < \langle \eta \rangle^{\tilde{q}^{1/2}}$, write
		\begin{align}\label{eqn:fourier_reduction}
			\|\nu_{y,\eta} - \chi_{y,\eta}\|_{\mathcal{S}^{-s}}
			= \|(\widehat{\chi_{y,\eta}} \mu_B - \widehat{\chi_{y,\eta}})^{\vee}\|_{\mathcal{S}^{-s}}
			\lesssim_s \|\widehat{\chi_{y,\eta}} \mu_B - \widehat{\chi_{y,\eta}}\|_{\mathcal{S}^{-s}}.
		\end{align}
		The Fourier transform $\widehat{\chi_{y,\eta}}$ is a wave packet adapted to $B_1(y') \times B_1(\eta')$ in phase space, where $(y',\eta') = (\eta,-y)$. Now $\langle y' \rangle^{\tilde{q}^{1/2}} \geq \langle \eta' \rangle^{\tilde{p}^{1/2}}$ and $B$ is $(q,\delta)$-dense, so we can apply the same argument as in the first case (but with $p,q$ swapped and $\tilde{p},\tilde{q}$ swapped) to estimate the right hand side of \eqref{eqn:fourier_reduction}.
	\end{proof}
	
	Take $s,k,\delta$ to be as in Proposition \ref{prop:nu_approx}, so \eqref{eqn:nu_approx} holds. Define $P \colon \mathcal{S}(\mathbf{R}^d) \to \mathcal{M}_{A,B}$ by
	\begin{align*}
		Pf = \int_{\mathbf{R}^d \times \mathbf{R}^d} Tf(y,\eta) \nu_{y,\eta} \d y \d\eta.
	\end{align*}
	This is our parametrix. By Gabor inversion and \eqref{eqn:nu_approx},
	\begin{align*}
		\|(1-P)f\|_{\mathcal{S}^{-s}}
		\leq \frac{1}{2} \|f\|_{\mathcal{S}^{-s}}
	\end{align*}
	for all $f \in \mathcal{S}(\mathbf{R}^d)$. Therefore $P$ extends continuously to a bounded linear operator on $\mathcal{S}^{-s}(\mathbf{R}^d)$, and the power series
	\begin{align*}
		\sum_{j=0}^{\infty} (1-P)^j
	\end{align*}
	converges in operator norm. This series inverts $P$, so $P$ surjects onto $\mathcal{S}^{-s}(\mathbf{R}^d)$. But $P$ maps into $\mathcal{M}_{A,B}$, so we must have $\mathcal{M}_{A,B} \supseteq \mathcal{S}^{-s}(\mathbf{R}^d)$, and hence $\mathcal{M}_{A,B} = \mathcal{S}'(\mathbf{R}^d)$ because $\mathcal{M}_{A,B}$ is closed and $\mathcal{S}^{-s}(\mathbf{R}^d)$ is dense. We conclude by duality that $\mathcal{F}_{A,B}$ is injective.
	
	\subsection{Completing the proof of Theorem \ref{thm:super_non-end}} \label{subsec:proof_super_non-end} Suppose now that $\delta$ is arbitrary. We must show that $\Ker \mathcal{F}_{A,B}$ has dimension $O_{\delta}(1)$, and $\Coker \mathcal{F}_{A,B}$ is infinite-dimensional. This is very similar to Section \ref{subsec:completing_sub_non-end}.
	
	To prove the first statement, add $O_{\delta}(1)$ points to $A,B$ to form supersets $A' \supseteq A$ and $B' \supseteq B$ which are $(p',\delta')$-dense and $(q',\delta')$-dense, respectively, for fixed $p' < p$ and $q' < q$ with $p'q' > 1$, and for $\delta'$ sufficiently small depending on $p',q'$. Then $\mathcal{F}_{A',B'}$ is injective by our work so far. It follows directly that $\dim \Ker \mathcal{F}_{A,B} = O_{\delta}(1)$.
	
	To prove the second statement, remove infinitely many points from $A,B$ to form subsets $A' \subseteq A$ and $B' \subseteq B$ which are $(p,\delta')$-dense and $(q,\delta')$-dense, respectively, for some $\delta' \lesssim \delta$. The subspace of $\mathcal{S}(A) \oplus \mathcal{S}(B)$ consisting of sequences which vanish on $A'$ and $B'$ is infinite-dimensional. Since $\Ker \mathcal{F}_{A',B'}$ is finite-dimensional, the image of $\mathcal{F}_{A,B}$ must have finite-dimensional intersection with this subspace. Therefore the image of $\mathcal{F}_{A,B}$ has infinite codimension in $\mathcal{S}(A) \oplus \mathcal{S}(B)$, i.e., $\Coker \mathcal{F}_{A,B}$ is infinite-dimensional.
	
	
	\section{Dense critical case: proof of Theorem \ref{thm:super_end}} \label{sec:super_end}
	
	
	\subsection{Oscillatory integral estimates} In this section, we carry out the usual integration by parts for oscillatory integrals of the form
	\begin{align*}
		\int_{\mathbf{R}^d} \psi e^{i\phi}
	\end{align*}
	with non-stationary phase (i.e., $\nabla\phi \neq 0$), paying attention to the dependence on derivatives of $\phi,\psi$.
	
	Since $\nabla\phi \neq 0$, we can integrate by parts by writing
	\begin{align*}
		\int \psi e^{i\phi}
		= \frac{1}{i}\int \psi \frac{\nabla \phi}{|\nabla\phi|^2} \cdot \nabla e^{i\phi}
		= i \int \Div\Big(\psi \frac{\nabla\phi}{|\nabla\phi|^2}\Big) e^{i\phi}.
	\end{align*}
	Let $\mathcal{D}_{\phi}$ denote the differential operator
	\begin{align*}
		\mathcal{D}_{\phi} \psi
		= i \Div\Big(\psi \frac{\nabla \phi}{|\nabla\phi|^2}\Big).
	\end{align*}
	Iterating the above,
	\begin{align}\label{eqn:D_phi_parts}
		\int \psi e^{i\phi}
		= \int (\mathcal{D}_{\phi}^n \psi) e^{i\phi}
	\end{align}
	for any $n \in \mathbf{N}$. We now want to estimate $\mathcal{D}_{\phi}^n \psi$.
	
	\begin{lem} \label{lem:D_phi_general}
		One has the pointwise bound
		\begin{align*}
			|\mathcal{D}_{\phi}^n \psi|
			\lesssim_n \sum_{k=0}^n \frac{1}{|\nabla \phi|^{n+k}} \sum_{\substack{j_1,\dots,j_k \geq 2 \\ j_1 + \cdots + j_k \leq n+k}} |\nabla^{j_1} \phi| \cdots |\nabla^{j_k} \phi| |\nabla^{n+k-j_1 - \cdots - j_k} \psi|
		\end{align*}
		(when $k=0$ we take the inner sum to be equal to $|\nabla^n\psi|$).
	\end{lem}
	
	\begin{proof}
		This could easily be proven by induction. Instead we give a proof which suggests how to guess the right hand side.
		
		By the product and quotient rules, $\mathcal{D}_{\phi}^n \psi$ is a linear combination of terms of the form
		\begin{align}\label{eqn:D_phi_term}
			\frac{(\partial^{\alpha_1}\phi) \cdots (\partial^{\alpha_k} \phi)}{|\nabla \phi|^{\ell}} \partial^{\beta}\psi
		\end{align}
		for some $k,\ell \in \mathbf{N}$ and multiindices $\alpha_1,\dots,\alpha_k,\beta \in \mathbf{N}^d$ (here $\ell$ will be even, but we won't use this for our estimates).
		By homogeneity considerations, we can constrain these parameters. Suppose $\phi$ is homogeneous of order $r$. By the product rule, $\mathcal{D}_{\phi}^n \psi$ is a linear combination of terms of the form
		\begin{align}\label{eqn:D_phi_term'}
			\partial^{\gamma_1} \Big(\frac{\partial_{i_1} \phi}{|\nabla\phi|^2}\Big) \cdots \partial^{\gamma_n} \Big(\frac{\partial_{i_n} \phi}{|\nabla \phi|^2}\Big) \partial^{\beta}\psi
		\end{align}
		for some indices $i_1,\dots,i_n \in \{1,\dots,d\}$ and multiindices $\gamma_1,\dots,\gamma_n,\beta \in \mathbf{N}^d$ with $|\gamma_1| + \cdots + |\gamma_n| + |\beta| = n$. The terms \eqref{eqn:D_phi_term} come from further expanding this. The part of $\eqref{eqn:D_phi_term}$ depending on $\phi$ is homogeneous of order
		\begin{align*}
			kr - |\alpha_1| - \cdots - |\alpha_k| - \ell(r-1),
		\end{align*}
		while the part of \eqref{eqn:D_phi_term'} depending on $\phi$ is homogeneous of order
		\begin{align*}
			n(1-r) - |\gamma_1| - \cdots - |\gamma_n|
			= |\beta| - nr.
		\end{align*}
		These orders must be equal for all $r$, so we must have
		\begin{align*}
			\ell = n+k
			\qquad \text{and} \qquad
			|\beta| = n+k - |\alpha_1| - \cdots - |\alpha_k|.
		\end{align*}
		Denoting $j_i = |\alpha_i|$, we conclude that $\mathcal{D}_{\phi}^n\psi$ is bounded by a sum of terms of the form
		\begin{align*}
			\frac{|\nabla^{j_1}\phi| \cdots |\nabla^{j_k}\phi|}{|\nabla\phi|^{n+k}} |\nabla^{n+k-j_1 - \cdots - j_k} \psi|
		\end{align*}
		with $n+k-j_1-\cdots - j_k \geq 0$. We may assume $j_1,\dots,j_k \geq 2$, as otherwise we can cancel some factors in the numerator and the denominator (and this cancellation doesn't change the shape of the term). Then the condition $j_1 + \cdots + j_k \leq n+k$ forces $k \leq n$. This completes the proof.
	\end{proof}
	
	We give two applications, Corollary \ref{cor:trivial_non-stationary} and Proposition \ref{prop:osc_int_main}. The first doesn't require such precise estimates as Lemma \ref{lem:D_phi_general} gives, but the second does.
	
	\begin{cor} \label{cor:trivial_non-stationary}
		Let $M,r > 0$. Suppose $\phi$ obeys the derivative bounds
		\begin{align*}
			|\nabla\phi(x)| \gtrsim M \langle x \rangle^r
			\qquad \text{and} \qquad
			|\nabla^k \phi(x)| \lesssim_k M \langle x \rangle^r
		\end{align*}
		for all $k \geq 2$ and $x \in \supp \psi$. Then for any $n \in \mathbf{N}$,
		\begin{align*}
			\Big|\int_{\mathbf{R}^d} \psi e^{i\phi}\Big|
			\lesssim_n M^{-n} \sum_{i=0}^n \int_{\mathbf{R}^d} |\nabla^i\psi(x)| \langle x \rangle^{-rn} \d x.
		\end{align*}
	\end{cor}
	
	\begin{proof}
		This is immediate from \eqref{eqn:D_phi_parts} and Lemma \ref{lem:D_phi_general}.
	\end{proof}
	
	\begin{prop}\label{prop:osc_int_main}
		Let $X \geq 1$, let $R,V,\lambda,M > 0$ with $R \lesssim X$, and let $r \in \mathbf{R}$. Suppose $\psi$ satisfies
		\begin{align*}
			\|\nabla^k\psi\|_{L^1(\mathbf{R}^d)} \lesssim_k VR^{-k}
		\end{align*}
		for all $k \geq 0$, and $\phi$ satisfies
		\begin{align*}
			|\nabla\phi(x)| \gtrsim \lambda
			\qquad \text{and} \qquad
			|\nabla^k \phi(x)| \lesssim_k M X^{r+1-k}
		\end{align*}
		for all $k \geq 2$ and $x \in \supp\psi$. Then
		\begin{align*}
			\Big|\int_{\mathbf{R}^d} \psi e^{i\phi}\Big|
			\lesssim V \Big\langle \frac{\lambda R}{\langle \lambda^{-1} M R X^{r-1} \rangle} \Big\rangle^{-\infty}.
		\end{align*}
	\end{prop}
	
	The model scenario is when $\psi$ is essentially constant at scales $\ll R$, the total mass of $\psi$ is $\lesssim V$, the support of $\psi$ is contained in $\{|x| \sim X\}$, and $\phi$ is a linear function plus the dot product of $m$ with an $(r+1)$-pseudohomogeneous diffeomorphism, where $m \in \mathbf{R}^d$ is a fixed vector of magnitude $\lesssim M$ (e.g. \eqref{eqn:phi_exmp}).
	
	\begin{proof}
		The trivial bound is $V$, so we can remove the outer angle brackets, and we just need to prove that
		\begin{align*}
			\Big|\int_{\mathbf{R}^d} \psi e^{i\phi}\Big|
			\lesssim V \Big(\frac{\langle \lambda^{-1} M R X^{r-1} \rangle}{\lambda R}\Big)^{\infty}.
		\end{align*}
		By \eqref{eqn:D_phi_parts}, Lemma \ref{lem:D_phi_general}, and the assumed derivative bounds,
		\begin{align*}
			\Big|\int \psi e^{i\phi}\Big|
			\lesssim_n V \sum_{k=0}^{n} \frac{1}{\lambda^{n+k}} \sum_{\substack{j_1,\dots,j_k \geq 2 \\ j_1 + \cdots + j_k \leq n+k}} M^k X^{k(r+1) - j_1 - \cdots - j_k} R^{-n-k + j_1 + \cdots + j_k}
		\end{align*}
		for any $n \in \mathbf{N}$. Since $R \lesssim X$, the inner sum is dominated by the term where $j_1 + \cdots + j_k$ is as small as possible, namely $2k$. Therefore
		\begin{align*}
			\Big|\int \psi e^{i\phi}\Big|
			&\lesssim_n V \sum_{k=0}^{n} \frac{1}{\lambda^{n+k}} M^k X^{k(r-1)} R^{-n+k}
			\\&= V \Big(\frac{1}{\lambda R}\Big)^n \sum_{k=0}^{n} (\lambda^{-1} M R X^{r-1})^k
			\\&\lesssim_n V \Big(\frac{\langle \lambda^{-1} M R X^{r-1} \rangle}{\lambda R}\Big)^n.
			\qedhere
		\end{align*}
	\end{proof}
	
	\subsection{Injectivity of $\mathcal{F}_{A,B}$: balanced density case} \label{subsec:balanced}
	
	We first prove Theorem \ref{thm:super_end} in the special case where $p = q = 1$, so $A,B$ have comparable density. This allows for some technical simplifications, although the main ideas are the same in the general case.
	
	By composing $\sigma_A,\sigma_B$ with $x \mapsto (-x_1,x_2,\dots,x_d)$ if necessary, we may assume $\sigma_A,\sigma_B$ are orientation-preserving.
	
	We proceed as in the proof of Theorem \ref{thm:super_non-end}, but instead of using Lemma \ref{lem:full_mu} to define $\mu_A$ and $\mu_B$, we explicitly put
	\begin{align*}
		\mu_A = \sum_{n \in \mathbf{Z}^d} \delta^d \det D\sigma_A(n) \Dirac_{\delta \sigma_A(n)}
		\qquad \text{and} \qquad
		\mu_B = \sum_{n \in \mathbf{Z}^d} \delta^d \det D\sigma_B(n) \Dirac_{\delta \sigma_B(n)}
	\end{align*}
	(we write $\Dirac_x$ instead of $\delta_x$ for a Dirac mass at $x$ to avoid overloading the symbol $\delta$).
	
	\begin{lem}[Twisted Poisson summation formula] \label{lem:Poisson_twisted}
		Let $\sigma \colon \mathbf{R}^d \to \mathbf{R}^d$ be an orientation-preserving $r$-pseudohomogeneous diffeomorphism for some $r > 0$. Then for all Schwartz functions $f \in \mathcal{S}(\mathbf{R}^d)$,
		\begin{align*}
			\sum_{n \in \mathbf{Z}^d} \det D\sigma(n) f(\sigma(n))
			= \sum_{m \in \mathbf{Z}^d} \int_{\mathbf{R}^d} f(x) e^{-2\pi im \cdot \sigma^{-1}(x)} \d x.
		\end{align*}
	\end{lem}
	
	The pseudohomogeneity condition could be relaxed --- it is only there to ensure that $f \circ \sigma$ is Schwartz.
	
	\begin{proof}
		By Poisson summation and the change of variables formula,
		\begin{align*}
			\sum_{n \in \mathbf{Z}^d} \det D\sigma(n) f(\sigma(n))
			= \sum_{m \in \mathbf{Z}^d} \widehat{(\det D\sigma) f \circ \sigma}(m)
			&= \sum_{m \in \mathbf{Z}^d} \int_{\mathbf{R}^d} f(\sigma(x)) e^{-2\pi im \cdot x} \det D\sigma(x) \d x
			\\&= \sum_{m \in \mathbf{Z}^d} \int_{\mathbf{R}^d} f(x) e^{-2\pi im \cdot \sigma^{-1}(x)} \d x.
			\qedhere
		\end{align*}
	\end{proof}
	
	Let
	\begin{align*}
		\nu_{y,\eta} =
		\begin{cases}
			\chi_{y,\eta} \mu_A &\mbox{if } \langle y \rangle \geq \langle \eta \rangle, \\
			(\widehat{\chi_{y,\eta}} \mu_B)^{\vee} &\mbox{if } \langle y \rangle < \langle \eta \rangle.
		\end{cases}
	\end{align*}
	
	By the same parametrix iteration procedure as in Section \ref{subsec:injectivity_non-end}, one deduces injectivity of $\mathcal{F}_{A,B}$ from the following proposition.
	
	\begin{prop}
		Assume $s$ is sufficiently large, and $\delta$ is sufficiently small depending on $s$. Then
		\begin{align}\label{eqn:nu_approx_balanced}
			\|\nu_{y,\eta} - \chi_{y,\eta}\|_{\mathcal{S}^{-s}} \leq \frac{1}{2} \langle y \rangle^{-s} \langle \eta \rangle^{-s}.
		\end{align}
	\end{prop}
	
	\begin{proof}
		First suppose $\langle y \rangle \geq \langle \eta \rangle$, so $\nu_{y,\eta} = \chi_{y,\eta} \mu_A$. We need pointwise bounds on
		\begin{align}\label{eqn:end_portrait}
			|T(\nu_{y,\eta} - \chi_{y,\eta})(z,\zeta)|
			= |\langle \nu_{y,\eta} - \chi_{y,\eta}, \chi_{z,\zeta} \rangle|
			= \Big|\int \chi_{y,\eta} \overline{\chi_{z,\zeta}} \d\mu_A - \int_{\mathbf{R}^d} \chi_{y,\eta} \overline{\chi_{z,\zeta}}\Big|.
		\end{align}
		When $|\zeta| \gtrsim \delta^{-\varepsilon} \langle y \rangle^{1+\varepsilon}$ (for fixed $\varepsilon > 0$), the trivial bound
		\begin{align}\label{eqn:end_portrait_triv_bd}
			|T(\nu_{y,\eta} - \chi_{y,\eta})(z,\zeta)|
			\lesssim \langle y - z \rangle^{-\infty}
		\end{align}
		will suffice, so suppose $|\zeta| \lesssim \delta^{-\varepsilon} \langle y \rangle^{1+\varepsilon}$. By Lemma \ref{lem:Poisson_twisted} and the definition of $\mu_A$, we can rewrite \eqref{eqn:end_portrait} as
		\begin{align*}
			|T(\nu_{y,\eta} - \chi_{y,\eta})(z,\zeta)|
			= \Big|\sum_{0 \neq m \in \mathbf{Z}^d} \int_{\mathbf{R}^d} \chi_{y,\eta}(x) \overline{\chi_{z,\zeta}(x)} e^{-2\pi im \cdot \sigma_A^{-1}(\delta^{-1}x)} \d x \Big|.
		\end{align*}
		Rescaling by $\delta$,
		\begin{align*}
			|T(\nu_{y,\eta} - \chi_{y,\eta})(z,\zeta)|
			= \delta^d \Big|\sum_{m \neq 0} \int_{\mathbf{R}^d} \chi_{y,\eta}(\delta x) \overline{\chi_{z,\zeta}(\delta x)} e^{-2\pi im \cdot \sigma_A^{-1}(x)} \d x\Big|.
		\end{align*}
		Denote $\tau_A = \sigma_A^{-1}$. Then $\tau_A$ is $2$-pseudohomogeneous. Expanding the definitions of $\chi_{y,\eta}$ and $\chi_{z,\zeta}$,
		\begin{align*}
			|T(\nu_{y,\eta} - \chi_{y,\eta})(z,\zeta)|
			= \delta^d \Big|\sum_{m \neq 0} \int_{\mathbf{R}^d} \chi(\delta x - y) \overline{\chi(\delta x - z)} e^{-2\pi i(\delta(\zeta - \eta) \cdot x + m \cdot \tau_A(x))} \d x\Big|.
		\end{align*}
		Let
		\begin{align*}
			I = \Big|\int_{\mathbf{R}^d} \chi(\delta x - y) \overline{\chi(\delta x - z)} e^{-2\pi i \phi(x)} \d x\Big|
		\end{align*}
		denote the absolute value of the integral, where
		\begin{align} \label{eqn:phi_exmp}
			\phi(x) = \delta(\zeta-\eta) \cdot x + m \cdot \tau_A(x)
		\end{align}
		is the phase. The gradient of the phase is
		\begin{align}\label{eqn:grad_phase_balanced}
			\nabla \phi(x)
			= \delta(\zeta - \eta) + D\tau_A(x)^T m.
		\end{align}
		When $|m| \gtrsim \delta^{1-2\varepsilon} \langle y \rangle^{1+2\varepsilon}$, it follows from the pseudohomogeneity bound on $|D\tau_A(x)^{-1}|$ and the inequalities $\langle \eta \rangle \leq \langle y \rangle$ and $|\zeta| \lesssim \delta^{-\varepsilon} \langle y \rangle^{1+\varepsilon}$ that the term $D\tau_A(x)^Tm$ dominates, so the norm of the gradient of the phase is $\gtrsim |m|\langle x \rangle$, and non-stationary phase (Corollary \ref{cor:trivial_non-stationary}) gives
		\begin{align}\label{eqn:end_portrait_large_m}
			I \lesssim |m|^{-\infty} \langle y \rangle^{-\infty} \langle z \rangle^{-\infty}.
		\end{align}
		Suppose instead that $|m| \lesssim \delta^{1-2\varepsilon} \langle y \rangle^{1+2\varepsilon}$. Since $m$ is integral and nonzero, $|m| \geq 1$, so we must have $|y| \gg \delta^{-0.9}$ (the exponent $0.9$ could be made arbitrarily close to $1$ by decreasing $\varepsilon$, but this is unnecessary).
		Now for $x$ near $\delta^{-1}y$, we can bound $|\nabla\phi(x)|$ below by $|\nabla\phi(\delta^{-1}y)|$. By the fundamental theorem of calculus and the $2$-pseudohomogeneity of $\tau_A$,
		\begin{align*}
			|\nabla\phi(x) - \nabla\phi(\delta^{-1}y)|
			= |D\tau_A(x)^Tm - D\tau_A(\delta^{-1}y)^Tm|
			\lesssim |m| |x - \delta^{-1}y|.
		\end{align*}
		Thus if
		\begin{align*}
			|x-\delta^{-1}y|
			\ll |m|^{-1} |\nabla\phi(\delta^{-1} y)|,
		\end{align*}
		then
		\begin{align*}
			|\nabla\phi(x) - \nabla\phi(\delta^{-1}y)|
			\ll |\nabla\phi(\delta^{-1}y)|
		\end{align*}
		and hence
		\begin{align*}
			|\nabla\phi(x)| \sim |\nabla\phi(\delta^{-1}y)|.
		\end{align*}
		Using Proposition \ref{prop:osc_int_main} (non-stationary phase) on the region
		\begin{align*}
			\{x \in \mathbf{R}^d : |x| \sim \delta^{-1}|y| \text{ and } |x-\delta^{-1}y| \ll |m|^{-1} |\nabla\phi(\delta^{-1}y)|\}
		\end{align*}
		with parameters
		\begin{align*}
			X &= \delta^{-1} |y|, \\
			R &= \min\{\delta^{-1}, |m|^{-1} |\nabla\phi(\delta^{-1}y)|\}, \\
			V &= \delta^{-d} \langle y-z \rangle^{-\infty}, \\
			\lambda &= |\nabla\phi(\delta^{-1}y)|, \\
			M &= |m|, \\
			r &= 1
		\end{align*}
		(the second argument in the minimum in the definition of $R$ is necessary to capture the derivative bounds of the implicit smooth cutoff to $\{|x-\delta^{-1}y| \ll |m|^{-1} |\nabla\phi(\delta^{-1}y)|\}$), and using the trivial bound on the complementary region, we obtain the estimate
		\begin{align*}
			I \lesssim V \Big\langle \frac{\lambda R}{\langle \lambda^{-1} M R X^{r-1} \rangle} \Big\rangle^{-\infty} + \delta^{-d} \langle y-z \rangle^{-\infty} \langle \delta |m|^{-1} \nabla \phi(\delta^{-1}y) \rangle^{-\infty}
			+ |y|^{-\infty} \langle z \rangle^{-\infty}.
		\end{align*}
		With our choice of parameters, the second term on the right hand side dominates the other two, so
		\begin{align}\label{eqn:end_portrait_main}
			I \lesssim \delta^{-d} \langle y-z \rangle^{-\infty} \langle \delta |m|^{-1} \nabla\phi(\delta^{-1}y) \rangle^{-\infty}
			= \delta^{-d} \langle y-z \rangle^{-\infty} \langle \delta^2 |m|^{-1} (\zeta - \eta + \delta^{-1} D\tau_A(\delta^{-1}y)^Tm) \rangle^{-\infty}.
		\end{align}
		We remark that the dual lattice $L^{\vee}$ mentioned in Section \ref{subsubsec:super_end} is $\delta^{-1} D\tau_A(\delta^{-1}y)^T \mathbf{Z}^d$.
		
		Combining \eqref{eqn:end_portrait_main}, \eqref{eqn:end_portrait_large_m}, and \eqref{eqn:end_portrait_triv_bd}, we find that
		\begin{align*}
			\|\nu_{y,\eta} - \chi_{y,\eta}\|_{\mathcal{S}^{-s}}
			&= \int_{\mathbf{R}^d \times \mathbf{R}^d} |T(\nu_{y,\eta} - \chi_{y,\eta})(z,\zeta)| \langle z \rangle^{-s} \langle \zeta \rangle^{-s} \d z \d\zeta
			\\&\lesssim \int_{\{(z,\zeta) : |\zeta| \lesssim \delta^{-\varepsilon} \langle y \rangle^{1+\varepsilon}\}} \Big[\langle y-z \rangle^{-\infty} \sum_{0 \neq |m| \lesssim \delta^{1-2\varepsilon} \langle y \rangle^{1+2\varepsilon}} \langle \delta^2 |m|^{-1} (\zeta - \eta + \delta^{-1} D\tau_A(\delta^{-1}y)^Tm) \rangle^{-\infty}
			\\&\qquad\qquad\qquad\qquad\qquad + \delta^d \langle y \rangle^{-\infty} \langle z \rangle^{-\infty} \sum_{|m| \gtrsim \delta^{1-2\varepsilon} \langle y \rangle^{1+2\varepsilon}} |m|^{-\infty}\Big] \langle z \rangle^{-s} \langle \zeta \rangle^{-s} \d z \d\zeta
			\\&\qquad + \int_{\{(z,\zeta) : |\zeta| \gtrsim \delta^{-\varepsilon} \langle y \rangle^{1+\varepsilon}\}} \langle y-z \rangle^{-\infty} \langle z \rangle^{-s} \langle \zeta \rangle^{-s} \d z \d\zeta.
		\end{align*}
		When $|y| \lesssim \delta^{-0.9}$, the first line of the right hand side vanishes, and we get
		\begin{align*}
			\|\nu_{y,\eta} - \chi_{y,\eta}\|_{\mathcal{S}^{-s}}
			\lesssim_s \delta^d \langle y \rangle^{-\infty} + \delta^{\varepsilon(s-d)} \langle y \rangle^{-(1+\varepsilon)(s-d)} \langle y \rangle^{-s}.
		\end{align*}
		When $|y| \gtrsim \delta^{-0.9}$, it follows from the $2$-pseudohomogeneity of $\tau_A$ and the fact that $\langle \eta \rangle \leq \langle y \rangle$ that
		\begin{align*}
			|\delta^{-1}D\tau_A(\delta^{-1}y)^Tm - \eta|
			\sim \delta^{-1} |D\tau_A(\delta^{-1}y)^Tm|
			\sim \delta^{-2}|m| \langle y \rangle.
		\end{align*}
		Therefore
		\begin{align*}
			\|\nu_{y,\eta} - \chi_{y,\eta}\|_{\mathcal{S}^{-s}}
			&\lesssim_s \sum_{0 \neq |m| \lesssim \delta^{1-2\varepsilon} \langle y \rangle^{1+2\varepsilon}} (\delta^{-2} |m|)^d (\delta^{-2} |m|\langle y \rangle)^{-s} \langle y \rangle^{-s}
			+ \delta^d \langle y \rangle^{-\infty}
			+ \delta^{\varepsilon(s-d)} \langle y \rangle^{-(1+\varepsilon)(s-d)} \langle y \rangle^{-s}
			\\&\lesssim_s \delta^{2(s-d)} \langle y \rangle^{-2s} + \delta^d \langle y \rangle^{-\infty} + \delta^{\varepsilon(s-d)} \langle y \rangle^{-(2+\varepsilon)s + (1+\varepsilon)d}.
		\end{align*}
		In both cases $|y| \lesssim \delta^{-0.9}$ and $|y| \gtrsim \delta^{-0.9}$, using once again that $\langle \eta \rangle \leq \langle y \rangle$, we conclude the desired bound \eqref{eqn:nu_approx_balanced} when $s$ is large and $\delta$ is sufficiently small depending on $s$.
		
		So far we have only considered the situation where $\langle y \rangle \geq \langle \eta \rangle$. As in the proof of Proposition \ref{prop:nu_approx} though, the case $\langle y \rangle < \langle \eta \rangle$ reduces to an estimate which is analogous to what we have already shown when $\langle y \rangle \geq \langle \eta \rangle$.
	\end{proof}
	
	\subsection{A variably rescaled Gabor transform} \label{subsec:modified_Gabor}
	
	As discussed at the end of Section \ref{subsubsec:super_end}, for the proof of Theorem \ref{thm:super_end} in general, we need to use a phase space basis $(\rho_{y,\eta})_{y,\eta \in \mathbf{R}^d}$, where $\rho_{y,\eta}$ is a wave packet adapted to $B_{R_{y,\eta}}(y) \times B_{R_{y,\eta}^{-1}}(\eta)$ in phase space, and where we have some flexibility to choose the radii $R_{y,\eta}$. Proposition \ref{prop:variably_rescaled_Gabor} below describes a modification of the Gabor transform which implements the decomposition of a function in this basis.
	
	A simple case, which we will use in the proof of the proposition, is when $R_{y,\eta} = R$ is constant. 
	Fix an $L^2$-normalized Schwartz function $\chi_R$ adapted to the ball $\{|x| \lesssim R\}$. Let $\chi_R$ have $L^2$ norm exactly $1$. For $y,\eta \in \mathbf{R}^d$, set
	\begin{align*}
		\chi_{R,y,\eta}(x) = \chi_R(x-y) e^{2\pi i\eta \cdot x}.
	\end{align*}
	Define a rescaled Gabor transform $T_R$ in terms of $\chi_R$ exactly as in Section \ref{subsec:Gabor}, so
	\begin{align*}
		T_Rf(y,\eta) = \langle f,\chi_{R,y,\eta} \rangle.
	\end{align*}
	The proof of the Gabor inversion formula (Proposition \ref{prop:Gabor_inversion}) only depended on $\chi$ having $L^2$ norm $1$, so the same inversion formula holds for $T_R$:
	\begin{align}\label{eqn:rescaled_Gabor_inversion}
		f = \int_{\mathbf{R}^d \times \mathbf{R}^d} T_Rf(y,\eta) \chi_{R,y,\eta} \d y \d\eta.
	\end{align}
	This achieves the goal of decomposing $f$ into wave packets $\chi_{R,y,\eta}$ adapted to $B_R(y) \times B_{R^{-1}}(\eta)$ in phase space.
	
	%
	%
	
	\begin{rem} \label{rem:R_y,eta}
		In the balanced density case $p = q = 1$, the ordinary Gabor transform was sufficient. This should be viewed as a lucky coincidence. In general, we need to choose $R_{y,\eta}$ carefully when $\langle y \rangle^{p^{1/2}} \sim \langle \eta \rangle^{q^{1/2}}$, in which case the optimal choice is
		\begin{align} \label{eqn:R_y_eta_key_regime}
			R_{y,\eta} \sim \langle y \rangle^{\frac{1-p}{2}}
			\qquad \text{and} \qquad
			R_{y,\eta}^{-1} \sim \langle \eta \rangle^{\frac{1-q}{2}}
		\end{align}
		(note $\langle y \rangle^{\frac{1-p}{2}} \langle \eta \rangle^{\frac{1-q}{2}} \sim 1$ because $\langle y \rangle^{p^{1/2}} \sim \langle \eta \rangle^{q^{1/2}}$ and $pq = 1$). When $p = q = 1$, the right hand sides in \eqref{eqn:R_y_eta_key_regime} reduce to $1$, so the ordinary Gabor transform works.
		
		In the proof of Theorem \ref{thm:super_end}, we will take
		\begin{align} \label{eqn:R_y_eta_def}
			R_{y,\eta} = \frac{\langle y \rangle^{1/2}}{\langle \eta \rangle^{1/2}}
		\end{align}
		for all $y,\eta \in \mathbf{R}^d$. This agrees with \eqref{eqn:R_y_eta_key_regime} when $\langle y \rangle^{p^{1/2}} \sim \langle \eta \rangle^{q^{1/2}}$. One can keep \eqref{eqn:R_y_eta_def} in mind as a model example in the proof of Proposition \ref{prop:variably_rescaled_Gabor}.
	\end{rem}
	
	\begin{prop}\label{prop:variably_rescaled_Gabor}
		For $y,\eta \in \mathbf{R}^d$, let $R_{y,\eta} > 0$ satisfy the sublinear growth/decay property
		\begin{align*}
			R_{y,\eta} \lesssim \langle y \rangle^{1-\varepsilon}
			\qquad \text{and} \qquad
			R_{y,\eta}^{-1} \lesssim \langle \eta \rangle^{1-\varepsilon}
		\end{align*}
		($\varepsilon > 0$ fixed), as well as the radial constancy property
		\begin{align*}
			R_{y,\eta} \sim R_{y',\eta'}
			\qquad \text{whenever} \qquad
			\langle y \rangle \sim \langle y' \rangle \text{ and } \langle \eta \rangle \sim \langle \eta' \rangle.
		\end{align*}
		Then there are $L^2$-normalized vector-valued wave packets $\rho_{y,\eta},\tilde{\rho}_{y,\eta} \in \mathcal{S}(\mathbf{R}^d; \ell^2)$ adapted to $B_{R_{y,\eta}}(y) \times B_{R_{y,\eta}^{-1}}(\eta)$ in phase space such that if $\widetilde{T}$ denotes the modified Gabor transform 
		\begin{align*}
			\widetilde{T}f(y,\eta) = \langle f,\tilde{\rho}_{y,\eta} \rangle
		\end{align*}
		for scalar-valued $f$ (so $\widetilde{T}f$ is vector-valued), then the inversion formula
		\begin{align}\label{eqn:variably_rescaled_inversion}
			f = \int_{\mathbf{R}^d \times \mathbf{R}^d} \widetilde{T}f(y,\eta) \cdot \rho_{y,\eta} \d y \d\eta
		\end{align}
		holds.
	\end{prop}
	
	
	All of our definitions for scalar-valued functions have obvious analogs for vector-valued functions. For example, the statement that $\rho_{y,\eta}$ is a vector-valued wave packet means that it is of the form \eqref{eqn:wave_packet_def}, where $\chi \in \mathcal{S}(\mathbf{R}^d;\ell^2)$ is a standard vector-valued Schwartz function in the sense that it satisfies the bounds \eqref{eqn:std_Schwartz_def}.
	Linear operators defined on scalar-valued functions act component-wise on vector-valued functions.
	The pointwise absolute value of a vector-valued function is its pointwise $\ell^2$ norm. From this it is clear how to extend all the usual function space norms, like $\|\cdot\|_{L^p(\mathbf{R}^d)}$ or $\|\cdot\|_{\mathcal{S}^{-s}(\mathbf{R}^d)}$, to the vector-valued setting. Multiplication of scalar-valued functions corresponds to dot product of vector-valued functions, and the inequality $|fg| \leq |f||g|$ for scalar-valued $f,g$ corresponds to Cauchy--Schwarz. Because of all these parallels, vector-valued functions are as nice as scalar-valued functions for our purposes.
	
	The sublinear growth/decay condition on $R_{y,\eta}$ in Proposition \ref{prop:variably_rescaled_Gabor} ensures that $\rho_{y,\eta}$ really is localized near $(y,\eta)$ in phase space. Otherwise the uncertainty principle would make $\rho_{y,\eta}$ indistinguishable from a Schwartz function adapted to a ball in either physical or frequency space (as discussed at the end of Section \ref{subsubsec:adapted}).
	
	The radial constancy condition could be relaxed to a local constancy condition, but we won't bother because \eqref{eqn:R_y_eta_def} is radial.
	
	\begin{proof}
		For dyadic numbers $N,M \in 2^{\mathbf{N}}$, let $P_{N,M}$ be the operator acting on functions $f$ on $\mathbf{R}^d$ by first applying a Littlewood--Paley projection to frequencies $\{\langle \xi \rangle \sim M\}$, and then applying the analogous physical space projection to $\{\langle x \rangle \sim N\}$. Then $P_{N,M}$ is essentially a projection to the region
		\begin{align}\label{eqn:P_N,M_image}
			\{(x,\xi) \in \mathbf{R}^d \times \mathbf{R}^d : \langle x \rangle \sim N \text{ and } \langle \xi \rangle \sim M\}
		\end{align}
		in phase space. Since $N,M \gtrsim 1$, and hence $NM \gtrsim 1$, this respects the uncertainty principle. We now have a preliminary decomposition
		\begin{align*}
			f = \sum_{N,M \in 2^{\mathbf{N}}} P_{N,M}f.
		\end{align*}
		For each $N,M \in 2^{\mathbf{N}}$, pick some $y_N,\eta_M \in \mathbf{R}^d$ with $\langle y_N \rangle \sim N$ and $\langle \eta_M \rangle \sim M$, and set $R_{N,M} \sim R_{y_N,\eta_M}$. By the radial constancy property, this is well-defined up to constants. Using the rescaled Gabor inversion formula \eqref{eqn:rescaled_Gabor_inversion}, further decompose
		\begin{align}
			f &= \sum_{N,M \in 2^{\mathbf{N}}} \int_{\mathbf{R}^d \times \mathbf{R}^d} T_{R_{N,M}} P_{N,M}f(y,\eta) \chi_{R_{N,M},y,\eta} \d y \d\eta
			\notag
			\\&= \sum_{N,M \in 2^{\mathbf{N}}} \int_{\mathbf{R}^d \times \mathbf{R}^d} \langle f, P_{N,M}^{\ast} \chi_{R_{N,M}, y,\eta} \rangle \chi_{R_{N,M},y,\eta} \d y \d\eta.
			\label{eqn:further_decomp}
		\end{align}
		Note that the adjoint $P_{N,M}^{\ast}$ is still essentially a projection to the region \eqref{eqn:P_N,M_image} in phase space (it acts by projecting in physical space first and then in frequency space).
		By the sublinear growth/decay property, the essential phase space support of the wave packet $\chi_{R_{N,M},y,\eta}$ is far from \eqref{eqn:P_N,M_image} unless $N \sim \langle y \rangle$ and $M \sim \langle \eta \rangle$.
		It follows by routine estimates that we can write
		\begin{align}\label{eqn:P_chi_decay}
			P_{N,M}^{\ast} \chi_{R_{N,M},y,\eta}
			= c_{y,\eta,N,M} \varphi_{y,\eta,N,M},
		\end{align}
		where $\varphi_{y,\eta,N,M}$ is an $L^2$-normalized wave packet adapted to $B_{R_{N,M}}(y) \times B_{R_{N,M}^{-1}}(\eta)$ in phase space, and $c_{y,\eta,N,M} \geq 0$ is a constant satisfying
		\begin{align} \label{eqn:c_rapid_decay}
			c_{y,\eta,N,M}
			\lesssim \Delta_{N,\langle y \rangle}^{\infty} \Delta_{M,\langle \eta \rangle}^{\infty};
		\end{align}
		here we denote
		\begin{align*}
			\Delta_{X,Y} = \min\Big\{\frac{X}{Y}, \frac{Y}{X}\Big\}
		\end{align*}
		(so $c_{y,\eta,N,M}$ is essentially the indicator that $N \sim \langle y \rangle$ and $M \sim \langle \eta \rangle$).
		Let $\rho_{y,\eta}, \tilde{\rho}_{y,\eta} \in \mathcal{S}(\mathbf{R}^d; \ell^2(2^{\mathbf{N}} \times 2^{\mathbf{N}}))$ be the vector-valued functions whose $(N,M)$ components are
		\begin{align*}
			\rho_{y,\eta,N,M} = c_{y,\eta,N,M}^{1/2} \chi_{R_{N,M},y,\eta}
			\qquad \text{and} \qquad
			\tilde{\rho}_{y,\eta,N,M}
			= c_{y,\eta,N,M}^{-1/2} P_{N,M}^{\ast} \chi_{R_{N,M},y,\eta}.
		\end{align*}
		By \eqref{eqn:P_chi_decay} and \eqref{eqn:c_rapid_decay}, only the $O(1)$ many components where $N \sim \langle y \rangle$ and $M \sim \langle \eta \rangle$ essentially contribute to $\rho_{y,\eta}, \tilde{\rho}_{y,\eta}$. Therefore $\rho_{y,\eta}, \tilde{\rho}_{y,\eta}$ are wave packets adapted to $B_{R_{y,\eta}}(y) \times B_{R_{y,\eta}^{-1}}(\eta)$ in phase space, as desired. The inversion formula \eqref{eqn:variably_rescaled_inversion} is immediate from \eqref{eqn:further_decomp}.
	\end{proof}
	
	Using the transform $\widetilde{T}$, we define a scale of Banach spaces between $\mathcal{S}(\mathbf{R}^d)$ and $\mathcal{S}'(\mathbf{R}^d)$ as in Section \ref{subsec:Gabor}. Let $\widetilde{\mathcal{S}}^{-s}(\mathbf{R}^d)$ for $s \in \mathbf{R}$ be the space of tempered distributions $f \in \mathcal{S}'(\mathbf{R}^d)$ for which the norm
	\begin{align*}
		\|f\|_{\widetilde{\mathcal{S}}^{-s}} = \int_{\mathbf{R}^d \times \mathbf{R}^d} |\widetilde{T}f(y,\eta)| \langle y \rangle^{-s} \langle \eta \rangle^{-s} \d y \d\eta
	\end{align*}
	is finite. The inclusions $\mathcal{S}(\mathbf{R}^d) \hookrightarrow \widetilde{\mathcal{S}}^{-s}(\mathbf{R}^d) \hookrightarrow \mathcal{S}'(\mathbf{R}^d)$ are continuous and dense.
	
	If $\psi_{y,\eta}$ is an $L^2$-normalized wave packet adapted to $B_{R_{y,\eta}}(y) \times B_{R_{y,\eta}^{-1}}(\eta)$ in phase space, then $\widetilde{T} \psi_{y,\eta}$ is essentially supported on $B_{R_{y,\eta}}(y) \times B_{R_{y,\eta}^{-1}}(\eta)$, and $\|\widetilde{T} \psi_{y,\eta}\|_{L^{\infty}} \lesssim 1$. Therefore
	\begin{align}\label{eqn:rho_y_eta_S_norm}
		\|\psi_{y,\eta}\|_{\widetilde{S}^{-s}}
		\lesssim_s \langle y \rangle^{-s} \langle \eta \rangle^{-s}.
	\end{align}
	
	
	\begin{prop}\label{prop:fourier_bdd_general}
		Let the notation be as above, and suppose in addition that $R_{\eta,-y} \sim R_{y,\eta}^{-1}$. Then the Fourier transform is bounded on $\widetilde{\mathcal{S}}^{-s}(\mathbf{R}^d)$.
	\end{prop}
	
	Since $R_{\eta,-y} \sim R_{\eta,y}$ by the radial constancy property, we could just as well have written $R_{\eta,y} \sim R_{y,\eta}^{-1}$. However, this condition comes from the fact that the Fourier transform acts on phase space by $(y,\eta) \mapsto (\eta,-y)$ (see Section \ref{subsubsec:adapted}), so we feel that writing $R_{\eta,-y}$ is more suggestive.
	
	The hypothesis $R_{\eta,-y} \sim R_{y,\eta}^{-1}$ is satisfied in the main example \eqref{eqn:R_y_eta_def}.
	
	\begin{proof}
		By the inversion formula \eqref{eqn:variably_rescaled_inversion} and the triangle inequality,
		\begin{align*}
			\|\hat{f}\|_{\widetilde{\mathcal{S}}^{-s}}
			\leq \int_{\mathbf{R}^d \times \mathbf{R}^d} |\widetilde{T}f(y,\eta)| \|\widehat{\rho_{y,\eta}}\|_{\widetilde{\mathcal{S}}^{-s}} \d y \d\eta,
		\end{align*}
		so it suffices to show that
		\begin{align}\label{eqn:rho_y_eta_hat_S_norm}
			\|\widehat{\rho_{y,\eta}}\|_{\widetilde{\mathcal{S}}^{-s}}
			\lesssim_s \langle y \rangle^{-s} \langle \eta \rangle^{-s}.
		\end{align}
		This follows from \eqref{eqn:rho_y_eta_S_norm}, since $\widehat{\rho_{y,\eta}}$ is an $L^2$-normalized wave packet adapted to $B_{R_{y,\eta}^{-1}}(\eta) \times B_{R_{y,\eta}}(-y)$ in phase space, and $R_{y,\eta}^{-1} \sim R_{\eta,-y}$.
	\end{proof}
	
	\subsection{Injectivity of $\mathcal{F}_{A,B}$: general case}
	
	The strategy is the same as for the balanced density case in Section \ref{subsec:balanced}, except that we use the modified Gabor transform $\widetilde{T}$ of Proposition \ref{prop:variably_rescaled_Gabor} with
	\begin{align*}
		R_{y,\eta}
		= \frac{\langle y \rangle^{1/2}}{\langle \eta \rangle^{1/2}}.
	\end{align*}
	We correspondingly use the vector-valued wave packets $\rho_{y,\eta}, \tilde{\rho}_{y,\eta}$ and the function spaces $\widetilde{\mathcal{S}}^{-s}(\mathbf{R}^d)$ from Section \ref{subsec:modified_Gabor}.
	
	As in Section \ref{subsec:balanced}, we may assume $\sigma_A,\sigma_B$ are orientation-preserving. Put
	\begin{align*}
		\mu_A = \sum_{n \in \mathbf{Z}^d} \delta^d \det D\sigma_A(n) \Dirac_{\delta \sigma_A(n)}
		\qquad \text{and} \qquad
		\mu_B = \sum_{n \in \mathbf{Z}^d} \delta^d \det D\sigma_B(n) \Dirac_{\delta \sigma_B(n)}.
	\end{align*}
	Let
	\begin{align*}
		\nu_{y,\eta} =
		\begin{cases}
			\rho_{y,\eta} \mu_A &\mbox{if } \langle y \rangle^{p^{1/2}} \geq \langle \eta \rangle^{q^{1/2}}, \\
			(\widehat{\rho_{y,\eta}} \mu_B)^{\vee} &\mbox{if } \langle y \rangle^{p^{1/2}} < \langle \eta \rangle^{q^{1/2}}
		\end{cases}
	\end{align*}
	(so $\nu_{y,\eta}$ is vector-valued). Then as usual, injectivity of $\mathcal{F}_{A,B}$ follows from the following proposition.
	
	\begin{prop}
		Assume $s$ is sufficiently large, and $\delta$ is sufficiently small depending on $s$. Then
		\begin{align} \label{eqn:nu_approx_general}
			\|\nu_{y,\eta} - \rho_{y,\eta}\|_{\widetilde{\mathcal{S}}^{-s}}
			\leq \frac{1}{2} \langle y \rangle^{-s} \langle \eta \rangle^{-s}.
		\end{align}
	\end{prop}
	
	\begin{proof}
		First consider the case $\langle y \rangle^{p^{1/2}} \geq \langle \eta \rangle^{q^{1/2}}$. Then $\nu_{y,\eta} = \rho_{y,\eta} \mu_A$. We need pointwise bounds on
		\begin{align}\label{eqn:super_end_general_ptwise}
			|\widetilde{T}(\nu_{y,\eta} - \rho_{y,\eta})(z,\zeta)|
			= |\langle \nu_{y,\eta} - \rho_{y,\eta}, \tilde{\rho}_{z,\zeta} \rangle|
			= \Big|\int \rho_{y,\eta} \cdot \overline{\tilde{\rho}_{z,\zeta}} \d\mu_A - \int_{\mathbf{R}^d} \rho_{y,\eta} \cdot \overline{\tilde{\rho}_{z,\zeta}} \Big|.
		\end{align}
		The trivial bound is
		\begin{align}\label{eqn:triv_bd_general}
			|\widetilde{T}(\nu_{y,\eta} - \rho_{y,\eta})(z,\zeta)|
			\lesssim \Big\langle \frac{y-z}{R_{y,\eta} + R_{z,\zeta}} \Big\rangle^{-\infty}.
		\end{align}
		We will apply this when $(z,\zeta)$ is in the complement of the region
		\begin{align*}
			\mathcal{R}
			= \{(z,\zeta) \in \mathbf{R}^d \times \mathbf{R}^d : |z-y| \lesssim \delta^{-\varepsilon} \langle y \rangle^{\frac{1}{2}+\varepsilon} \text{ and } |\zeta| \lesssim \delta^{-\varepsilon} \langle y \rangle^{p+\varepsilon}\}
		\end{align*}
		($\varepsilon > 0$ small and fixed).
		
		Suppose now that $(z,\zeta) \in \mathcal{R}$. Using the twisted Poisson summation formula (Lemma \ref{lem:Poisson_twisted}) to rewrite the right hand side of \eqref{eqn:super_end_general_ptwise},
		\begin{align*}
			|\widetilde{T}(\nu_{y,\eta} - \rho_{y,\eta})(z,\zeta)|
			= \Big|\sum_{0 \neq m \in \mathbf{Z}^d}\int_{\mathbf{R}^d} \rho_{y,\eta}(x) \cdot \overline{\tilde{\rho}_{z,\zeta}(x)} e^{-2\pi im \cdot \sigma_A^{-1}(\delta^{-1}x)} \d x\Big|.
		\end{align*}
		Rescaling by $\delta$,
		\begin{align*}
			|\widetilde{T}(\nu_{y,\eta} - \rho_{y,\eta})(z,\zeta)|
			= \delta^d \Big|\sum_{m \neq 0}\int_{\mathbf{R}^d} \rho_{y,\eta}(\delta x) \cdot \overline{\tilde{\rho}_{z,\zeta}(\delta x)} e^{-2\pi im \cdot \sigma_A^{-1}(x)} \d x\Big|.
		\end{align*}
		Denote $\tau_A = \sigma_A^{-1}$. Then $\tau_A$ is $(p+1)$-pseudohomogeneous. Write
		\begin{align*}
			\rho_{y,\eta}(x)
			= R_{y,\eta}^{-d/2} \rho\Big(\frac{x-y}{R_{y,
					\eta}}\Big) e^{2\pi i\eta \cdot x}
			\qquad \text{and} \qquad
			\tilde{\rho}_{z,\zeta}(x)
			= R_{z,\zeta}^{-d/2} \tilde{\rho}\Big(\frac{x-z}{R_{z,\zeta}}\Big)e^{2\pi i\zeta \cdot x},
		\end{align*}
		with $\rho,\tilde{\rho} \in \mathcal{S}(\mathbf{R}^d;\ell^2)$ standard vector-valued Schwartz functions. Then
		\begin{align}\label{eqn:sum_of_osc_ints}
			|\widetilde{T}(\nu_{y,\eta} - \rho_{y,\eta})(z,\zeta)|
			= \delta^d R_{y,\eta}^{-d/2} R_{z,\zeta}^{-d/2} \Big|\sum_{m \neq 0} \int_{\mathbf{R}^d} \rho\Big(\frac{\delta x - y}{R_{y,\eta}}\Big) \cdot \overline{\tilde{\rho}\Big(\frac{\delta x - z}{R_{z,\zeta}}\Big)} e^{-2\pi i(\delta(\zeta-\eta) \cdot x + m \cdot \tau_A(x))} \d x\Big|.
		\end{align}
		Let
		\begin{align*}
			I = \Big|\int_{\mathbf{R}^d} \rho\Big(\frac{\delta x - y}{R_{y,\eta}}\Big) \cdot \overline{\tilde{\rho}\Big(\frac{\delta x - z}{R_{z,\zeta}}\Big)} e^{-2\pi i \phi(x)} \d x \Big|
		\end{align*}
		denote the absolute value of the integral, where
		\begin{align*}
			\phi(x) = \delta(\zeta-\eta) \cdot x + m \cdot \tau_A(x)
		\end{align*}
		is the phase. The gradient of the phase is
		\begin{align}\label{eqn:grad_phase_general}
			\nabla\phi(x)
			= \delta(\zeta-\eta) + D\tau_A(x)^Tm.
		\end{align}
		
		We need bounds on $I$ for $(z,\zeta,m) \in \mathcal{R} \times (\mathbf{Z}^d \setminus \{0\})$. We split this into three regimes, each of which is treated differently:
		\begin{align*}
			\mathcal{R} \times (\mathbf{Z}^d \setminus \{0\})
			= \mathcal{R}_1 \cup \mathcal{R}_2 \cup \mathcal{R}_3,
		\end{align*}
		where
		\begin{align*}
			\mathcal{R}_1 &= \{(z,\zeta,m) \in \mathcal{R} \times (\mathbf{Z}^d \setminus \{0\}) : |m| \gtrsim \delta^{1-2\varepsilon} \langle y \rangle^{p+2\varepsilon}\}, \\
			\mathcal{R}_2 &= \{(z,\zeta,m) \in \mathcal{R} \times (\mathbf{Z}^d \setminus \{0\}) : |m| \lesssim \delta^{1-2\varepsilon} \langle y \rangle^{p+2\varepsilon} \text{ and } |\zeta| \not\sim \delta^{-p-1} |m| \langle y \rangle^p\}, \\
			\mathcal{R}_3 &= \{(z,\zeta,m) \in \mathcal{R} \times (\mathbf{Z}^d \setminus \{0\}) : |m| \lesssim \delta^{1-2\varepsilon} \langle y \rangle^{p+2\varepsilon} \text{ and } |\zeta| \sim \delta^{-p-1} |m| \langle y \rangle^p\}.
		\end{align*}
		
		\begin{enumerate}
			\item[1.] In $\mathcal{R}_1$, we estimate $I$ by non-stationary phase everywhere.
			
			\item[2.] In $\mathcal{R}_2$, we estimate $I$ by non-stationary phase for $|x| \sim \delta^{-1} |y|$, and by Schwartz decay elsewhere.
			
			\item[3.] In $\mathcal{R}_3$, we estimate $I$ by non-stationary phase when $x$ is closer to $\delta^{-1}z$ than the stationary point, and by Schwartz decay elsewhere.
		\end{enumerate}
		The bulk of the mass of $\widetilde{T}(\nu_{y,\eta} - \rho_{y,\eta})$ comes from $\mathcal{R}_3$ when the stationary point is close to $\delta^{-1}z$, in which case the trivial bound for $I$ is sharp (up to inconsequential factors). Accordingly, the bounds for $I$ in the regime $\mathcal{R}_3$ are somewhat delicate, while in $\mathcal{R}_1$ and $\mathcal{R}_2$ they are straightforward.
		
		First, suppose $(z,\zeta,m) \in \mathcal{R}_1$. Then it follows from the $(p+1)$-pseudohomogeneity of $\tau_A$ and the inequalities $\langle \eta \rangle \leq \langle y \rangle^p$ and $|\zeta| \lesssim \delta^{-\varepsilon} \langle y \rangle^{p+\varepsilon}$ that the term $D\tau_A(x)^Tm$ in \eqref{eqn:grad_phase_general} dominates, so
		\begin{align*}
			|\nabla\phi(x)| \sim |D\tau_A(x)^Tm| \sim |m| \langle x \rangle^p
		\end{align*}
		everywhere on $\mathbf{R}^d$. Non-stationary phase (Corollary \ref{cor:trivial_non-stationary}) then yields
		\begin{align}\label{eqn:R_1_bound}
			I \lesssim |m|^{-\infty} \langle y \rangle^{-\infty} \langle z \rangle^{-\infty}.
		\end{align}
		
		Second, suppose $(z,\zeta,m) \in \mathcal{R}_2$. Then $1 \leq |m| \lesssim \delta^{1-2\varepsilon} \langle y \rangle^{p+2\varepsilon}$, so in particular $|y| \gtrsim \delta^{-\frac{1}{2p}}$. Now that $y$ is somewhat large, the condition $(z,\zeta) \in \mathcal{R}$ implies $|z| \sim |y|$. By Schwartz decay outside the region $\{x \in \mathbf{R}^d : |x| \sim \delta^{-1}|y|\}$,
		\begin{align*}
			I \lesssim \Big|\int_{|x| \sim \delta^{-1}|y|} \rho\Big(\frac{\delta x-y}{R_{y,\eta}}\Big) \cdot \overline{\tilde{\rho}\Big(\frac{\delta x - z}{R_{z,\zeta}}\Big)} e^{-2\pi i\phi(x)} \d x\Big| + |y|^{-\infty}.
		\end{align*}
		When $|x| \sim \delta^{-1}|y|$, we have $|D\tau_A(x)^Tm| \sim \delta^{-p} |m| |y|^p$, so the condition $|\zeta| \not\sim \delta^{-p-1} |m| \langle y \rangle^p$ in the definition of $\mathcal{R}_2$ combined with $\langle \eta \rangle \leq \langle y \rangle^p$ tells us that there is no cancellation in \eqref{eqn:grad_phase_general}. Therefore
		\begin{align*}
			|\nabla\phi(x)| \gtrsim |D\tau_A(x)^Tm| \sim |m| \langle x \rangle^p
		\end{align*}
		for $|x| \sim \delta^{-1}|y|$, and non-stationary phase (Corollary \ref{cor:trivial_non-stationary}) again yields
		\begin{align*}
			\Big|\int_{|x| \sim \delta^{-1}|y|} \rho\Big(\frac{\delta x-y}{R_{y,\eta}}\Big) \cdot \overline{\tilde{\rho}\Big(\frac{\delta x - z}{R_{z,\zeta}}\Big)} e^{-2\pi i\phi(x)} \d x\Big|
			\lesssim |m|^{-\infty} |y|^{-\infty}
		\end{align*}
		(note that the region $\{|x| \sim \delta^{-1}|y|\}$ is much larger than the essential support of the integrand, so the implicit smooth cutoff to this region has harmless derivative growth). Hence
		\begin{align}\label{eqn:R_2_bound}
			I \lesssim |y|^{-\infty}
			\lesssim \delta^{\infty} |m|^{-\infty} \langle y \rangle^{-\infty} \langle z \rangle^{-\infty}.
		\end{align}
		
		Third, and finally, suppose $(z,\zeta,m) \in \mathcal{R}_3$. Once again, $|y| \gtrsim \delta^{-\frac{1}{2p}}$ and $|z| \sim |y|$. By the definition of $\mathcal{R}_3$,
		\begin{align*}
			|\zeta| \sim \delta^{-p-1}|m| \langle y \rangle^p \geq \langle y \rangle^p \geq \langle \eta \rangle,
		\end{align*}
		so
		\begin{align}\label{eqn:R_z_zeta_size}
			R_{y,\eta}
			\gtrsim R_{z,\zeta}
			\sim \delta^{\frac{p+1}{2}} |m|^{-\frac{1}{2}} |y|^{\frac{1-p}{2}}.
		\end{align}
		Decompose $\mathbf{R}^d = \Omega_{\text{near}} \cup \Omega_{\text{far}}$, where
		\begin{align*}
			\Omega_{\text{near}}
			= \{x \in \mathbf{R}^d : |x-\delta^{-1}z| \ll \delta^{p-1} |m|^{-1} |y|^{1-p} |\nabla\phi(\delta^{-1}z)|\}
		\end{align*}
		and
		\begin{align*}
			\Omega_{\text{far}}
			= \{x \in \mathbf{R}^d : |x-\delta^{-1}z| \gtrsim \delta^{p-1} |m|^{-1} |y|^{1-p} |\nabla\phi(\delta^{-1}z)|\}.
		\end{align*}
		Qualitatively, $\Omega_{\text{near}}$ consists of points $x \in \mathbf{R}^d$ which are near $\delta^{-1}z$, while $\Omega_{\text{far}}$ consists of points which are far from $\delta^{-1}z$. The precise threshold $\delta^{p-1} |m|^{-1} |y|^{1-p} |\nabla\phi(\delta^{-1}z)|$ will be motivated presently. Note first that by the $(p+1)$-pseudohomogeneity of $\tau_A$, the inequality $\langle \eta \rangle \leq \langle y \rangle^p$, and the fact that we are in $\mathcal{R}_3$,
		\begin{align*}
			|\nabla\phi(\delta^{-1}z)|
			\lesssim \delta^{-p} |m| |y|^p,
		\end{align*}
		so
		\begin{align*}
			\Omega_{\text{near}} \subseteq \{x \in \mathbf{R}^d : |x| \sim \delta^{-1}|y|\}.
		\end{align*}
		Now for $|x| \sim \delta^{-1}|y|$, we can use the fundamental theorem of calculus and pseudohomogeneity to estimate
		\begin{align*}
			|\nabla\phi(x) - \nabla\phi(\delta^{-1}z)|
			= |D\tau_A(x)^Tm - D\tau_A(\delta^{-1}z)^Tm|
			\lesssim \delta^{1-p} |m||y|^{p-1} |x-\delta^{-1}z|.
		\end{align*}
		Thus for $x \in \Omega_{\text{near}}$,
		\begin{align*}
			|\nabla\phi(x) - \nabla\phi(\delta^{-1}z)|
			\ll |\nabla\phi(\delta^{-1}z)|,
		\end{align*}
		and hence
		\begin{align}\label{eqn:grad_near}
			|\nabla\phi(x)| \sim |\nabla\phi(\delta^{-1}z)|.
		\end{align}
		The estimate \eqref{eqn:grad_near} is the motivation for the precise definition of $\Omega_{\text{near}}$ --- we would like $\Omega_{\text{near}}$ to be as large as possible, but if we expanded it by even a large constant factor, then it could potentially contain a stationary point. Write $I \lesssim I_{\text{near}} + I_{\text{far}}$, where
		\begin{align*}
			I_{\text{near}}
			= \Big|\int_{\Omega_{\text{near}}} \rho\Big(\frac{\delta x - y}{R_{y,\eta}}\Big) \cdot \overline{\tilde{\rho}\Big(\frac{\delta x - z}{R_{z,\zeta}}\Big)} e^{-2\pi i\phi(x)} \d x\Big|,
		\end{align*}
		and $I_{\text{far}}$ is defined similarly. Estimate $I_{\text{far}}$ trivially:
		\begin{align*}
			I_{\text{far}}
			\lesssim \int_{\Omega_{\text{far}}} \Big\langle \frac{\delta x - y}{R_{y,\eta}} \Big\rangle^{-\infty} \Big\langle \frac{\delta x - z}{R_{z,\zeta}} \Big\rangle^{-\infty} \d x
			&\lesssim \delta^{-d} R_{z,\zeta}^d \Big\langle \frac{y-z}{R_{y,\eta}} \Big\rangle^{-\infty} \Big\langle \frac{\delta}{R_{z,\zeta}} \delta^{p-1}|m|^{-1}|y|^{1-p} \nabla\phi(\delta^{-1}z) \Big\rangle^{-\infty}
			\\&\sim \delta^{-d} R_{z,\zeta}^d \Big\langle \frac{y-z}{R_{y,\eta}} \Big\rangle^{-\infty} \langle \delta^{-1} R_{z,\zeta} \nabla\phi(\delta^{-1}z) \rangle^{-\infty}
		\end{align*}
		(here we've used \eqref{eqn:R_z_zeta_size} multiple times). Estimate $I_{\text{near}}$ by applying Proposition \ref{prop:osc_int_main} with the parameters
		\begin{align*}
			X &= \delta^{-1} |y|, \\
			R &= \min\{\delta^{-1} R_{z,\zeta}, \delta^{p-1} |m|^{-1} |y|^{1-p} |\nabla\phi(\delta^{-1}z)|\} \lesssim X, \\
			V &= \delta^{-d} R_{z,\zeta}^d \Big\langle \frac{y-z}{R_{y,\eta}} \Big\rangle^{-\infty}, \\
			\lambda &= |\nabla\phi(\delta^{-1}z)|, \\
			M &= |m|, \\
			r &= p
		\end{align*}
		(the second argument in the minimum in the definition of $R$ is necessary to capture the derivative bounds of the implicit smooth cutoff to $\Omega_{\text{near}}$). This yields the bound
		\begin{align*}
			I_{\text{near}}
			\lesssim V \Big\langle \frac{\lambda R}{\langle \lambda^{-1} M R X^{r-1} \rangle} \Big\rangle^{-\infty}.
		\end{align*}
		By \eqref{eqn:R_z_zeta_size} and our choice of parameters,
		\begin{align*}
			\lambda R \sim \min\{\delta^{-1} R_{z,\zeta} |\nabla\phi(\delta^{-1}z)|, (\delta^{-1} R_{z,\zeta} |\nabla\phi(\delta^{-1}z)|)^2\}
			\qquad \text{and} \qquad
			\lambda^{-1} M R X^{r-1} \lesssim 1,
		\end{align*}
		so
		\begin{align*}
			I_{\text{near}}
			\lesssim \delta^{-d} R_{z,\zeta}^d \Big\langle \frac{y-z}{R_{y,\eta}} \Big\rangle^{-\infty} \langle \delta^{-1} R_{z,\zeta} \nabla\phi(\delta^{-1}z) \rangle^{-\infty}.
		\end{align*}
		Not coincidentally, this is the same bound we got for $I_{\text{far}}$. Writing
		\begin{align*}
			\delta^{-1} \nabla\phi(\delta^{-1}z)
			= \zeta - \eta + \delta^{-1} D\tau_A(\delta^{-1}z)^T m,
		\end{align*}
		we conclude that
		\begin{align}\label{eqn:R_3_bound}
			I \lesssim I_{\text{near}} + I_{\text{far}}
			\lesssim \delta^{-d} R_{z,\zeta}^d \Big\langle \frac{y-z}{R_{y,\eta}} \Big\rangle^{-\infty} \langle R_{z,\zeta}(\zeta - \eta + \delta^{-1} D\tau_A(\delta^{-1} z)^Tm)\rangle^{-\infty}.
		\end{align}
		
		Combining our estimates \eqref{eqn:R_1_bound}, \eqref{eqn:R_2_bound}, and \eqref{eqn:R_3_bound} in the three regimes $\mathcal{R}_1, \mathcal{R}_2, \mathcal{R}_3$, we obtain the bound
		\begin{align*}
			I \lesssim |m|^{-\infty} \langle y \rangle^{-\infty} \langle z \rangle^{-\infty} + 1_{\Big\{\substack{|z| \sim |y| \gtrsim \delta^{-1/(2p)} \\ |\zeta| \sim \delta^{-p-1} |m| \langle y \rangle^p}\Big\}} \delta^{-d} R_{z,\zeta}^d \Big\langle \frac{y-z}{R_{y,\eta}} \Big\rangle^{-\infty} \langle R_{z,\zeta}(\zeta - \eta + \delta^{-1} D\tau_A(\delta^{-1} z)^Tm)\rangle^{-\infty}
		\end{align*}
		for all $(z,\zeta) \in \mathcal{R}$ and $m \in \mathbf{Z}^d \setminus \{0\}$. Plugging this into \eqref{eqn:sum_of_osc_ints},
		\begin{align}\label{eqn:T_nu_error_general}
			|\widetilde{T}(\nu_{y,\eta} - \rho_{y,\eta})(z,\zeta)|
			&\lesssim \delta^{(1-\frac{\varepsilon}{2})d} \langle y \rangle^{-\infty} \langle z \rangle^{-\infty}
			\\&\quad + \sum_{m \neq 0} 1_{\Big\{\substack{|z| \sim |y| \gtrsim \delta^{-1/(2p)} \\ |\zeta| \sim \delta^{-p-1} |m| \langle y \rangle^p}\Big\}} \Big(\frac{R_{z,\zeta}}{R_{y,\eta}}\Big)^{d/2} \Big\langle \frac{y-z}{R_{y,\eta}} \Big\rangle^{-\infty} \langle R_{z,\zeta}(\zeta-\eta + \delta^{-1} D\tau_A(\delta^{-1}z)^Tm) \rangle^{-\infty}.
			\notag
		\end{align}
		for $(z,\zeta) \in \mathcal{R}$.
		Note that in the relevant regime $|z| \sim |y| \gtrsim \delta^{-\frac{1}{2p}}$, we have
		\begin{align*}
			|\delta^{-1} D\tau_A(\delta^{-1}z)^Tm - \eta|
			\sim \delta^{-1} |D\tau_A(\delta^{-1}z)^Tm|
			\sim \delta^{-p-1} |m| |y|^p
		\end{align*}
		by pseudohomogeneity and the fact that $\langle \eta \rangle \leq \langle y \rangle^p$.
		Integrating \eqref{eqn:T_nu_error_general} over $\mathcal{R}$ against the weight $\langle z \rangle^{-s} \langle \zeta \rangle^{-s}$ which appears in the definition of $\widetilde{\mathcal{S}}^{-s}$,
		\begin{align*}
			\int_{\mathcal{R}} |\widetilde{T}(\nu_{y,\eta} - \rho_{y,\eta})(z,\zeta)| \langle z \rangle^{-s} \langle \zeta \rangle^{-s} \d z \d\zeta
			\lesssim_s \delta^{(1-\frac{\varepsilon}{2})d} \langle y \rangle^{-\infty}
			+ \delta^{(p+1)(s-d/4)} \frac{\langle y \rangle^{pd/4}}{\langle \eta \rangle^{d/4}} \langle y \rangle^{-(p+1)s}.
		\end{align*}
		On $\mathcal{R}^c$, we use the trivial bound \eqref{eqn:triv_bd_general} to estimate
		\begin{align*}
			\int_{\mathcal{R}^c} |\widetilde{T}(\nu_{y,\eta} - \rho_{y,\eta})(z,\zeta)| \langle z \rangle^{-s} \langle \zeta \rangle^{-s} \d z \d\zeta
			&\lesssim \int_{|z-y| \gtrsim \delta^{-\varepsilon} \langle y \rangle^{\frac{1}{2}+\varepsilon}} \Big\langle \frac{y-z}{R_{y,\eta} + R_{z,\zeta}} \Big\rangle^{-\infty} \langle z \rangle^{-s} \langle \zeta \rangle^{-s} \d z \d\zeta
			\\&\quad + \int_{|\zeta| \gtrsim \delta^{-\varepsilon} \langle y \rangle^{p+\varepsilon}} \Big\langle \frac{y-z}{R_{y,\eta} + R_{z,\zeta}} \Big\rangle^{-\infty} \langle z \rangle^{-s} \langle \zeta \rangle^{-s} \d z \d\zeta
			\\&\lesssim_s \delta^{\infty} \langle y \rangle^{-\infty} + R_{y,\eta}^d \langle y \rangle^{-s} (\delta^{-\varepsilon} \langle y \rangle^{p+\varepsilon})^{-(s-d)}.
		\end{align*}
		Summing the two bounds above, and using as always that $R_{y,\eta} \leq \langle y \rangle^{1/2}$ and $\langle \eta \rangle \leq \langle y \rangle^p$,
		\begin{align*}
			\|\nu_{y,\eta} - \rho_{y,\eta}\|_{\widetilde{S}^{-s}}
			\lesssim_s \Big[\delta^{(1-\frac{\varepsilon}{2})d} \langle y \rangle^{-\infty} + \delta^{(p+1)(s-d/4)} \frac{\langle y \rangle^{p(d/4-s)}}{\langle \eta \rangle^{d/4-s}} + \delta^{\varepsilon(s-d)} \langle y \rangle^{-\varepsilon(s-d) + (p+\frac{1}{2})d}\Big] \langle y \rangle^{-s} \langle \eta \rangle^{-s}.
		\end{align*}
		Fixing $\varepsilon$, taking $s$ large depending on $\varepsilon$, and then taking $\delta$ small depending on $s$, we finally obtain the desired bound \eqref{eqn:nu_approx_general}.
		
		It remains to treat the case where $\langle y \rangle^{p^{1/2}} < \langle \eta \rangle^{q^{1/2}}$. We proceed similarly to the end of the proof of Proposition \ref{prop:nu_approx}. In this case, $\nu_{y,\eta} = (\widehat{\rho_{y,\eta}} \mu_B)^{\vee}$, so by Proposition \ref{prop:fourier_bdd_general},
		\begin{align}\label{eqn:frequency_case_general}
			\|\nu_{y,\eta} - \rho_{y,\eta}\|_{\widetilde{\mathcal{S}}^{-s}}
			= \|(\widehat{\rho_{y,\eta}} \mu_B - \widehat{\rho_{y,\eta}})^{\vee}\|_{\widetilde{\mathcal{S}}^{-s}}
			\lesssim_s \|\widehat{\rho_{y,\eta}} \mu_B - \widehat{\rho_{y,\eta}}\|_{\widetilde{\mathcal{S}}^{-s}}.
		\end{align}
		Denote $(y',\eta') = (\eta,-y)$. Then $\widehat{\rho_{y,\eta}}$ is an $L^2$-normalized wave packet adapted to $B_{R_{y',\eta'}}(y') \times B_{R_{y',\eta'}^{-1}}(\eta')$ in phase space. Now $\langle y' \rangle^{q^{1/2}} \geq \langle \eta' \rangle^{p^{1/2}}$, so we can apply the above arguments with $p,q$ swapped to estimate the right hand side of \eqref{eqn:frequency_case_general}.
	\end{proof}
	
	\subsection{Completing the proof of Theorem \ref{thm:super_end}}
	
	It remains to check that $\Coker \mathcal{F}_{A,B}$ is infinite-dimensional. Consider the subsets
	\begin{align*}
		A' = \{\delta\sigma_A(2n) : n \in \mathbf{Z}^d\} \subseteq A
		\qquad \text{and} \qquad
		B' = \{\delta\sigma_B(2n) : n \in \mathbf{Z}^d\} \subseteq B.
	\end{align*}
	Since $\delta$ is small, our work above shows that $\mathcal{F}_{A',B'}$ is injective. Now apply the same argument as in Section \ref{subsec:proof_super_non-end} to conclude.
	

	\parskip 0em
\end{document}